\documentclass{amsart}

\usepackage{todonotes}
\usepackage[normalem]{ulem}
\usepackage{biblatex}
\usepackage{geometry}
\usepackage{textgreek}

\newtheorem{theorem}{Theorem}[section]
\newtheorem{lemma}[theorem]{Lemma}
\newtheorem{corollary}[theorem]{Corollary}
\newtheorem{proposition}[theorem]{Proposition}
\theoremstyle{definition}

\theoremstyle{remark}
\newtheorem{remark}[theorem]{Remark}
\newtheorem{example}[theorem]{Example}

\numberwithin{equation}{section}

\newcommand{\F}{\mathcal{F}}

\newcommand{\B}{\mathcal{B}}
\renewcommand{\S}{\mathcal{S}} 
\newcommand{\R}{\mathbb{R}}

\newcommand{\N}{\mathbb{N}}
\newcommand{\weakly}{\rightharpoonup}
\renewcommand{\t}{\ dt} 
\newcommand{\s}{\ ds}
\renewcommand{\i}{\infty} 
\newcommand{\SO}{\mathrm{SO}}

\definecolor{dkblue}{rgb}{0,0,0.5}
\definecolor{dkgreen}{rgb}{0,0.5,0.0}
\definecolor{dkred}{rgb}{0.5,0,0.0}
\definecolor{dkpurp}{rgb}{0.5,0,0.5}
\definecolor{aqua}{rgb}{0.0, 0.8, 0.8}

\newcommand{\MJ}{\color{dkgreen}}
\newcommand{\EEE}{\color{black}}

\newcommand{\ACH}{\color{red}}

\usepackage{upgreek}

\bibliography{main.bib}
\begin{document}

\title[$\Gamma$-convergence of a discrete Kirchhoff rod energy]{$\pmb\Gamma$-convergence of a discrete Kirchhoff rod energy}

\author{Patrick Dondl}
\address{Patrick Dondl\\
Abteilung für Angewandte Mathematik\\
Albert-Ludwigs-Universität Freiburg\\
Hermann-Herder-Str.~10\\
79104 Freiburg i.~Br.}
\email{patrick.dondl@mathematik.uni-freiburg.de}

\author{Coffi Aristide Hounkpe}
\address{Coffi Aristide Hounkpe\\
Abteilung für Angewandte Mathematik\\
Albert-Ludwigs-Universität Freiburg\\
Hermann-Herder-Str.~10\\
79104 Freiburg i.~Br.}
\email{coffi.hounkpe@mathematik.uni-freiburg.de}

\author{Martin Jesenko}
\address{Martin Jesenko\\
Univerza v Ljubljani \\
Fakulteta za gradbeništvo in geodezijo \\
Jamova cesta 2 \\
1001 Ljubljana}
\curraddr{}
\email{martin.jesenko@fgg.uni-lj.si}
\thanks{}

\subjclass[2010]{49M25, 35A15, 74K10}

\date{\today}

\begin{abstract}
This work is motivated by the classical discrete elastic rod model by Audoly \textit{et al}. We derive a discrete version of the Kirchhoff elastic energy for rods undergoing bending and torsion and prove $\Gamma$-convergence to the continuous model. This discrete energy is given by the bending and torsion energy of an interpolating conforming polynomial curve and provides a simple formula for the bending energy depending in each discrete segment only on angle and adjacent edge lengths. For the $\liminf$-inequality, we need to introduce penalty terms to ensure arc-length parametrization in the limit. For the recovery sequence a  discretization with equal Euclidean  distance between consecutive points is constructed. Particular care is taken to treat the interaction between bending and torsion by employing a discrete version of the Bishop frame.
 
\end{abstract}
\keywords{Kirchhoff rods, Discrete rods, $\Gamma$-Convergence}
\maketitle

\tableofcontents

\section{Introduction}
In this work, we are concerned with the study of \emph{Kirchhoff rods}, i.e., slender, essentially one-dimensional elastic objects undergoing bending and twisting. The Kirchhoff rod model is commonly used in engineering applications and can be derived from fully nonlinear 3d elasticity as a slender body limit when assuming an isotropic material and a rotationally symmetric rod cross section as a special case in \cite{MoraMueller}.

A number of approaches have been proposed to numerically treat such rod theories. We mention for example work by Bartels \emph{et al.}~\cite{bartels2013simple}, where a $W^{2,2}$-conforming discretization for the bending energy was combined with a linearization approach for the inextensibility constraint for Euler-Bernoulli beams. In \cite{Bartels2} similar techniques were used to discretize rods including twist. For full Cosserat rods which also include the possibility for stretching and shearing, we refer to numerical work by, e.g., Jung \emph{et al.}~\cite{Jung_Leyendecker}. A particularly efficient approach to the numerical treatment of Kirchhoff rods was proposed by Bergou \emph{et al.}~\cite{bergou2008discrete}, further developed, e.g., in \cite{Lestringant,Korner}.

While for conforming discretizations energetic ($\Gamma$-)convergence results in the sense that discrete energies faithfully reproduce the continuum model are comparatively easy to obtain, some more work is needed to treat discrete rods since piecewise affine curves in general yield infinite bending energy. In light of \cite[Table 1]{Wardetzky}, we refer to \cite{Alibert1, Alibert2, , Bruckstein, Espanol, Iglesias} for $\Gamma$-convergence results, as well as to \cite{Wardetzky} for a convergence result in Hausdorff distance, for discrete beam theories all without twisting.

We also follow a $\Gamma$-convergence approach to rigorously derive a continuum limit for the full Kirchhoff theory including torsion. Constructing an equivalence between discrete (piecewise affine) rods and appropriate spline interpolations enables us to show convergence of torsion-free frames to their continuum limits. As in the original \cite{bergou2008discrete} this allows to include a scalar parameter describing twist with respect to this torsion-free frame. Care was taken to include physically relevant boundary conditions in our theory, where, e.g., the full frame may be prescribed on either end of the rod.

The remainder of this work is organized as follows. In section \ref{sec:modeling} we introduce the Kirchhoff rod theory and the torsion-free Bishop frame for both inextensible continuum rods as well as our discretization including a penalty ensuring arc-length parametrization in the limit. At the end of this section, we also state our main convergence result, Theorem \ref{theo:main}, and derive an equivalent energy that depends only on (next-)nearest neighbor interactions. Section \ref{sect:lowerbound} is concerned with the lower-bound estimate for $\Gamma$-convergence. Sections \ref{sect:app} and \ref{sect:recovery} are then finally devoted to the construction of a recovery sequence.

\section{Modeling inextensible elastic rods} \label{sec:modeling}


\subsection{Energy of a framed rod}

In \cite{MoraMueller}, Mora and Müller derive the energy of a framed rod in terms of $\Gamma$-convergence starting from nonlinear elasticity. We briefly recapitulate their results here.

Starting with a standard nonlinear elastic energy for a slender body on a reference domain $\Omega_h = (0,L) \times hS$ with cross section $S \subset \R^2$, they obtain --  by means of dimensional reduction -- a limiting energy for $h\to 0$ of the form
\begin{equation} \label{eq:energy-rod-general}
    \frac{1}{2} \int_{0}^{L} Q( \mathbf{R}(s)^{\top} \mathbf{R}'(s) ) \ ds
\end{equation}
with a quadratic form $Q : \R^{3 \times 3} \to \R $ if the deformed rod is described by a framed curve $(u,d_2,d_3) \in W^{2,2}((0,L);\R^3) \times W^{1,2}((0,L);\R^3) \times W^{1,2}((0,L);\R^3)$ with frame $(u'(s),d_2(s),d_3(s))=\mathbf{R}(s) \in \SO(3)$ for (almost) all $s\in (0,L)$.  Above, $\mathbf{R}'(s)$ denotes the derivative of $\mathbf{R}(s)$ with respect to $s$. The quadratic form $Q$ is derived by linearization and relaxation from the nonlinear elastic energy density.

We now note that $\mathbf{R}(s)^{\top} \mathbf{R}'(s)$ for $\mathbf{R} \in W^{1,2}((0,L); \SO(3))$ is a skew-symmetric matrix.  For any skew-symmetric matrix $ A \in \R^{3 \times 3} $, there exists an axial vector $ \omega \in \R^{3} $ such that $ Ax = \omega \times x $. If $ \omega(s) $ is the axial vector of $ \mathbf{R}(s)^{\top} \mathbf{R}'(s) $, then \eqref{eq:energy-rod-general} can be equivalently expressed as
\[  \frac{1}{2} \int_{0}^{L} q( \omega(s) ) \ ds \]
with a quadratic form $q : \R^{3} \to \R $.

In \cite[Remark 3.5]{MoraMueller}, for the case of an isotropic material elastic energy density, $q$ is explicitly calculated.
%
%
%
One obtains (after aligning the coordinates with principle axes of the second moment of inertia of the cross section $S$) the well-known expression for energy density of an inextensible rod, i.e.,
\[
q( \omega ) 
= G J_1 \omega_{1}^{2} 
+ E ( J_{2} \omega_{2}^{2} + J_{3} \omega_{3}^{2} ) 
\]
with $G$ being the shear modulus and $ E $ the Young modulus  of the material and with torsion  constant $J_1$ as well as moments of inertia $ J_{2} $ and $ J_{3} $ with respect to principal axes being properties of the cross section $S$.

For cross sections for which $J_2=J_3=J$, e.g., circular rods, we simply obtain 
\begin{equation}\label{eq:energy-rod-isotropic}
q(\omega) = G J_1 \omega_{1}^{2} 
+ E J \kappa^2 
\end{equation}
for the squared rod curvature $\kappa^2 = \omega_2^2+\omega_3^2$.

\subsection{Bishop frame}  \label{sec:bishp}
It is now convenient to introduce a torsion-free frame, the so-called \emph{Bishop frame} $ \mathbf{B}=(b_1,b_2,b_3) $. Its defining property is that the skew-symmetric matrix $\mathbf{B}(s)^{\top} \mathbf{B}'(s)$ takes the form
\[ \mathbf{B}(s)^{\top} \mathbf{B}'(s) = \begin{pmatrix} 0 & - k_{1}(s) & - k_{2}(s) \\ k_{1}(s) & 0 & 0 \\ k_{2}(s) & 0 & 0 \end{pmatrix}. \]
By \cite{bishop} there exists such a frame for a arc-length parametrized $C^2$-curve.

The matrix equation above can be rewritten as
\begin{align*}
    b_{1}'(s) &= k_1(s) b_{2}(s) + k_2(s) b_{3}(s), \\
    b_{2}'(s) &= - k_1(s) b_{1}(s), \\
    b_{3}'(s) &= - k_2(s) b_{1}(s).
\end{align*}
Automatically $ b_{1}'(s) = u''(s) \perp u'(s) = b_{1}(s) $.
Therefore, for any $s$ we are looking for an orthonormal basis for $ \{ b_{1}(s) \}^{\perp} $ such that its derivative has the direction $ b_{1}(s) $.
Since $ u''(s) \perp u'(s) $, there exists an unique $ \Omega(s) $ with $ \Omega(s) \perp u'(s) $ 
such that $ u''(s) = \Omega(s) \times u'(s) $, namely
\[ \Omega(s) = u'(s) \times u''(s). \]
If we decompose $ b_{1}'(s) = k_1(s) b_{2}(s) + k_2(s) b_{3}(s) $, with $ ( b_{2}(s) , b_{3}(s) ) $ still being an arbitrary orthonormal basis of $ \{ b_{1}(s) \}^{\perp} $, then
\[ \Omega(s) = k_{1}(s) b_{3}(s) - k_{2}(s) b_{2}(s). \]
Hence, the system for a Bishop frame reads
\begin{align*}
    b_{2}'(s) &= \Omega(s) \times b_{2}(s), \\
    b_{3}'(s) &= \Omega(s) \times b_{3}(s),
\end{align*}
noting that $u$ and thus $u'$, $u''$ and $ \Omega $ are given.

Thus, a Bishop frame can easily be constructed for an arc-length parametrized $W^{2,2}$-curve $u$, yielding a $W^{1,2}$-frame. Fixing such a Bishop frame $ \mathbf{B}$, any $W^{1,2}$-frame $ \mathbf{R} $ on the curve $u$ can be expressed by a rotation angle $ \theta $ -- a function in $W^{1,2}(0,L)$ -- of its second and third axis around $ u'(s)=b_1(s) $, i.e.
\begin{equation} 
\label{eq:arbitrary-frame}
\mathbf{R}(s) = \mathbf{B}(s) \mathbf{\Theta}^{\theta}(s) 
\quad \mbox{for} \quad 
 \mathbf{\Theta}^{\theta}(s)  = 
\begin{pmatrix} 1 & 0 & 0 \\ 0 & \cos \theta(s) & - \sin \theta(s) \\ 0 & \sin \theta(s) & \cos \theta(s) \end{pmatrix}.
\end{equation}
Employing this denotation, the set of all Bishop frames reads $ \{ \mathbf{B} \mathbf{\Theta}^{\theta} : \theta \mbox{ constant} \} $; the set of Bishop frames for a given curve is thus a one-parameter family of frames. Then
\begin{align*}
& \mathbf{R}(s)^{\top} \mathbf{R}'(s) \\
& =
\begin{pmatrix} 
0 & - k_{1}(s) \cos \theta(s) - k_{2}(s) \sin \theta(s) & k_{1}(s) \sin \theta(s) - k_{2}(s) \cos \theta(s) \\ 
k_{1}(s) \cos \theta(s) + k_{2}(s) \sin \theta(s) & 0 & - \theta'(s) \\ 
- k_{1}(s) \sin \theta(s) + k_{2}(s) \cos \theta(s) & \theta'(s)  & 0
\end{pmatrix}.
\end{align*} 
Its axial vector is
\[ \omega(s) = \begin{pmatrix} 
\theta'(s) \\ 
k_{1}(s) \sin \theta(s) - k_{2}(s) \cos \theta(s) \\ 
k_{1}(s) \cos \theta(s) + k_{2}(s) \sin \theta(s)
\end{pmatrix}, \]
where $k_1^2+k_2^2 = |u''|^2$. Thus, the energy density in \eqref{eq:energy-rod-isotropic} can be written as
\[ q(\omega) =  E J |u''|^2 + G J_{1} (\theta')^2. \]
We remark that in this particular case of an isotropic material and equal principal moments of inertia of the cross section the specific choice of Bishop's frame is immaterial.

In the following the goal will thus be to find a suitable, simple discrete approximation for the functional
\begin{equation}
\label{eq:continuous-energy} 
\F(u,\theta) :=
\left\{ \begin{array}{cl}
\int_{0}^{L} | u''(s) |^{2} +| \theta'(s) |^{2} \ ds & \mbox{if }(u,\theta) \in W^{2,2}( (0,L) ; \R^{3} ) \times W^{1,2}((0,L)) \mbox{ and } |u'| \equiv 1, \\
\i & \mbox{else.}
\end{array} \right.
\end{equation}
We write $\F(u,\theta) = \F^\textrm{bend}(u) + \F^\textrm{tor}(\theta)$.
Note that, for the sake of simplicity, we have set $EJ=G J_1=2$.

\subsection{Discrete rods}
As noted in the introduction, for numerical applications we want to write a discrete version of the energy derived above. Following \cite{bergou2008discrete}, we consider a rod described by an ordered set of points in $\R^3$. A simple piecewise affine approximation of our rod is then  given by the set of points with connecting straight edges (or segments) between points adjacent in the ordering.

As in the continuum case, we can now define a natural (``Bishop'') frame to associate with the discrete, piecewise affine rod. Again, we simply fix a frame on the first point describing the rod. This frame is then associated with the first edge. By rotating the frame in the plane given by two adjacent edges by the angle between the two edges, we can propagate this frame along the rod. This procedure only runs into trouble if two consecutive edges make a 180 degree turn. As, in this work, we will exclusively be concerned with discrete curves converging to arc-length parametrized $W^{2,2}$-curves we may from now on  ignore this issue -- it arises several times hereafter, but can always  only occur finitely often in any sequence of curves.

This allows us to again define a current frame on the discrete rod. To do this, we associate an angle with each edge. This angle now describes the deviation from the aforementioned Bishop frame on each edge.

A framed discrete rod is thus a pair $ ( X , \Phi ) \in ( \R^{3} )^{N+1} \times \R^{N} $ for some $ N \in \N $.
We write $ X = ( x_{0} , \ldots , x_{N} ) $ and $ \Phi = ( \varphi_{1} , \ldots , \varphi_{N} ) $.
The $ x_{i} $ prescribe the positions of points on the rod and $ \varphi_{i} $ angles relative to the Bishop frame for an edge between $ x_{i-1}$ and $ x_{i} $. As in the continuous case, it is enough to provide the angle by which the Bishop frame must be rotated around the axis given by the edge tangent to obtain the current rod frame.

The next step is to derive a reasonable energy functional for this discrete rod. The piecewise affine rod described above makes it difficult to directly give sense to a curvature energy. We thus introduce a suitable mapping from discrete to continuous and sufficiently smooth rods. This mapping will keep the important features of the discrete Bishop frame, as described above, intact.

\subsection{Cubic function determined by a triple}
\label{subsect:cubic}
We start by deriving a cubic polynomial suitable for locally approximating such a discrete curve. This is a simple procedure akin to fitting a Hermite polynomial; we provide the details here for the readers' convenience. Let us thus have $ x_{-1} , x_{0} , x_{1} \in \R^{3} $ given and determine the cubic polynomial $S\colon  \tau \mapsto S(\tau) \in \R^3$ for $\tau \in\R$ that for some $ T > 0 $ fulfils,
\begin{equation*}
S(0) = \frac{1}{2} ( x_{-1} + x_{0} ), \quad
S(T) = \frac{1}{2} ( x_{0} + x_{1} )
\end{equation*}
and
\begin{equation*}
S'(0) = \sigma_{0} \frac{ x_{0} - x_{-1} }{ | x_{0} - x_{-1} | }, \quad
S'(T) = \sigma_{1} \frac{ x_{1} - x_{0} }{ | x_{1} - x_{0} | }
\end{equation*}
with $ \sigma_{0} , \sigma_{1} \in \R $ given.

We can immediately fix the constant and the linear coefficient of the cubic polynomial to obtain
\[ 
S(\tau) = A \tau^{3} + B \tau^{2} + \sigma_{0} \frac{ x_{0} - x_{-1} }{ | x_{0} - x_{-1} | } \tau + \frac{1}{2} ( x_{-1} + x_{0} ).
\]
Due to the condition at $\tau=T$, the quadratic and cubic coefficients must satisfy
\begin{eqnarray*}
B T^{2} 
&  =  & \left( \tfrac{3}{2} - \tfrac{ \sigma_{1} T }{ | x_{1} - x_{0} | } \right) x_{1} 
		+ \left( \tfrac{ \sigma_{1} T }{ | x_{1} - x_{0} | } - \tfrac{ 2 \sigma_{0} T }{ | x_{0} - x_{-1} | } \right) x_{0} 
		+ \left( - \tfrac{3}{2} + \tfrac{ 2 \sigma_{0} T }{ | x_{0} - x_{-1} | } \right) x_{-1}, \\
A T^{3} 
&  =  & \left( -1 + \tfrac{ \sigma_{1} T }{ | x_{1} - x_{0} | } \right) x_{1} 
		+ \left( - \tfrac{ \sigma_{1} T }{ | x_{1} - x_{0} | } + \tfrac{ \sigma_{0} T }{ | x_{0} - x_{-1} | } \right) x_{0} 
		+ \left( 1 - \tfrac{ \sigma_{0} T }{ | x_{0} - x_{-1} | } \right) x_{-1}.
\end{eqnarray*}
If 
\begin{equation}
\label{eq:spline-1}
| x_{1} - x_{0} | = | x_{0} - x_{-1} | = r,
\end{equation} 
we obtain
\begin{eqnarray*}
B T^{2} 
&  =  & \left( \tfrac{3}{2} - \tfrac{ \sigma_{1} T }{ r } \right) x_{1} 
		+ \left( \tfrac{ \sigma_{1} T }{ r } - \tfrac{ 2 \sigma_{0} T }{ r } \right) x_{0} 
		+ \left( - \tfrac{3}{2} + \tfrac{ 2 \sigma_{0} T }{ r } \right) x_{-1}, \\
A T^{3} 
&  =  & \left( -1 + \tfrac{ \sigma_{1} T }{ r } \right) x_{1} 
		+ \left( - \tfrac{ \sigma_{1} T }{ r } + \tfrac{ \sigma_{0} T }{ r } \right) x_{0} 
		+ \left( 1 - \tfrac{ \sigma_{0} T }{ r } \right) x_{-1}.
\end{eqnarray*}
If moreover 
\begin{equation}
\label{eq:spline-2}
\sigma_{0} = \sigma_{1} = \sigma,
\end{equation} 
then
\begin{eqnarray*}
B T^{2} 
&  =  & \frac{ x_{1} - 2 x_{0} + x_{-1} }{2}
		+ \left( 1 - \frac{ \sigma T }{ r } \right)(  x_{1} + x_{0} - 2 x_{-1} ), \\
A T^{3} 
&  =  & \left( -1 + \frac{ \sigma T }{ r } \right) ( x_{1} - x_{-1} ).
\end{eqnarray*}
If -- even further -- it holds that
\begin{equation} 
\label{eq:spline-3}
r = \sigma T, 
\end{equation}
then  $S$ is a quadratic function and 
\begin{equation*}
B T^{2} 
=  \frac{ x_{1} - 2 x_{0} + x_{-1} }{2}.
\end{equation*}

\subsection{Assigning a framed spline to a set of points and angles}
\label{subsect:assign-spline}

Based on the local cubic spline curve developed above, we can now derive a suitable smooth spline interpolation for our discrete curve. For any ordered set of points 
$ X = ( x_{0} , \ldots , x_{N} ) \in ( \R^{3} )^{N+1} $
we denote by $ \ell(X) $ the length of the polygonal line through these points, i.e.,
\begin{equation*}
\ell(X) := | x_{1} - x_{0} | + \ldots + | x_{N} - x_{N-1} |.
\end{equation*}
Moreover, let 
\begin{equation*}
\lambda(X) := \frac{ \ell(X) }{L},
\end{equation*}
with $L$ fixed.
We now assign to $X$ a $ {\S}^{3,1} $-spline $ y^{X}: [0,L] \to \R^{3} $ that is an element of $ W^{2,2}( (0,L) ; \R^{3} ) $. First we introduce a partition of $ [0 , \ell(X) ] $ by setting $ \tau_{0} := 0 $, $ \tau_{N+1} := \ell(X) $ and for every $ i = 1 , \ldots , N $ 
\begin{equation} 
\label{eq:def-tau}
\tau_{i} :=  \sum_{k=1}^{i} | x_{k} - x_{k-1} | - \frac{1}{2}| x_{i} - x_{i-1} |.
\end{equation}
We define $ \eta^{X} : [0, \ell(X) ] \to \R^{3} $ by
\begin{equation} 
\label{eq:def-eta}
\eta^{X}( \tau ) 
:= \left\{ \begin{array}{cl}
x_{0} + \tau \frac{ x_{1} - x_{0} }{ | x_{1} - x_{0} | } , & \tau \le \tau_{1}, \\
S_{i}( \tau - \tau_{i} ), & \tau_{i} \le \tau \le \tau_{i+1} \mbox{ for some } i \in \{ 1 , \ldots, N-1 \}, \\
x_{N} - ( \ell(X) - \tau ) \frac{ x_{N} - x_{N-1} }{ | x_{N} - x_{N-1} | }, & \tau_{N} \le \tau .
\end{array} \right.
\end{equation}
$ S_{i} : [0, \tau_{i+1}-\tau_{i}] \to \R^{3} $ is the only cubic polynomial that satisfies:
\[ S_{i}(0) = \overline{ x_{i} }, \quad
S_{i}(\tau_{i+1}-\tau_{i}) = \overline{ x_{i+1} }, \quad
S_{i}'(0) = \frac{ x_{i} - x_{i-1} }{ | x_{i} - x_{i-1} | }, \quad
S_{i}'(\tau_{i+1}-\tau_{i}) =\frac{ x_{i+1} - x_{i} }{ | x_{i+1} - x_{i} | } \]
with $ \overline{ x_{i} } := \frac{1}{2}( x_{i-1} + x_{i} ) $ which we determined in section~\ref{subsect:cubic}.
Then we simply rescale
\begin{equation*}
y^{X}: [0,L] \to \R^{3}, \quad
y^{X}(t) := \eta^{X}( \lambda(X) t ). 
\end{equation*}
For the shape of this curve, see Figure~\ref{figure:spline}. We identify elements from $ (\R^{3} )^{N+1} $ with functions from $ W^{2,2}( (0,L) ; \R^{3} ) $ via $ X \equiv y^{X} $.

The function $ y^{X} $ is piecewise cubic. Denote by $ ( t_{0} , \ldots , t_{N+1} ) $ the partition  of $ [0,L] $ where 
$ t_{i} := \frac{\tau_{i} }{ \lambda(X) } $ 
corresponds to the distances between the consecutive points in $X$. \EEE
At the interior knots, i.e., for $ i \in \{ 1 , \ldots , N \} $, it holds 
\begin{equation}
\label{eq:properties-spline-knots}
y^{X}( t_{i} ) = \overline{ x_{i} }
\quad \mbox{and} \quad
(y^{X})'(t_{i} ) =\lambda(X) \frac{ x_{i} - x_{i-1} }{ | x_{i} - x_{i-1} | }
\end{equation}
whereas the first and the last part are linear with
\begin{equation}
\label{eq:properties-spline-linear}
(y^{X})'(t) =\lambda(X) \frac{ x_{1} - x_{0} }{ | x_{1} - x_{0} | } \quad \mbox{for $ t \le t_{1} $}
\end{equation}
and analogously for $ t \ge t_{N} $.

\begin{figure}[ht]
\includegraphics[width=.5\textwidth]{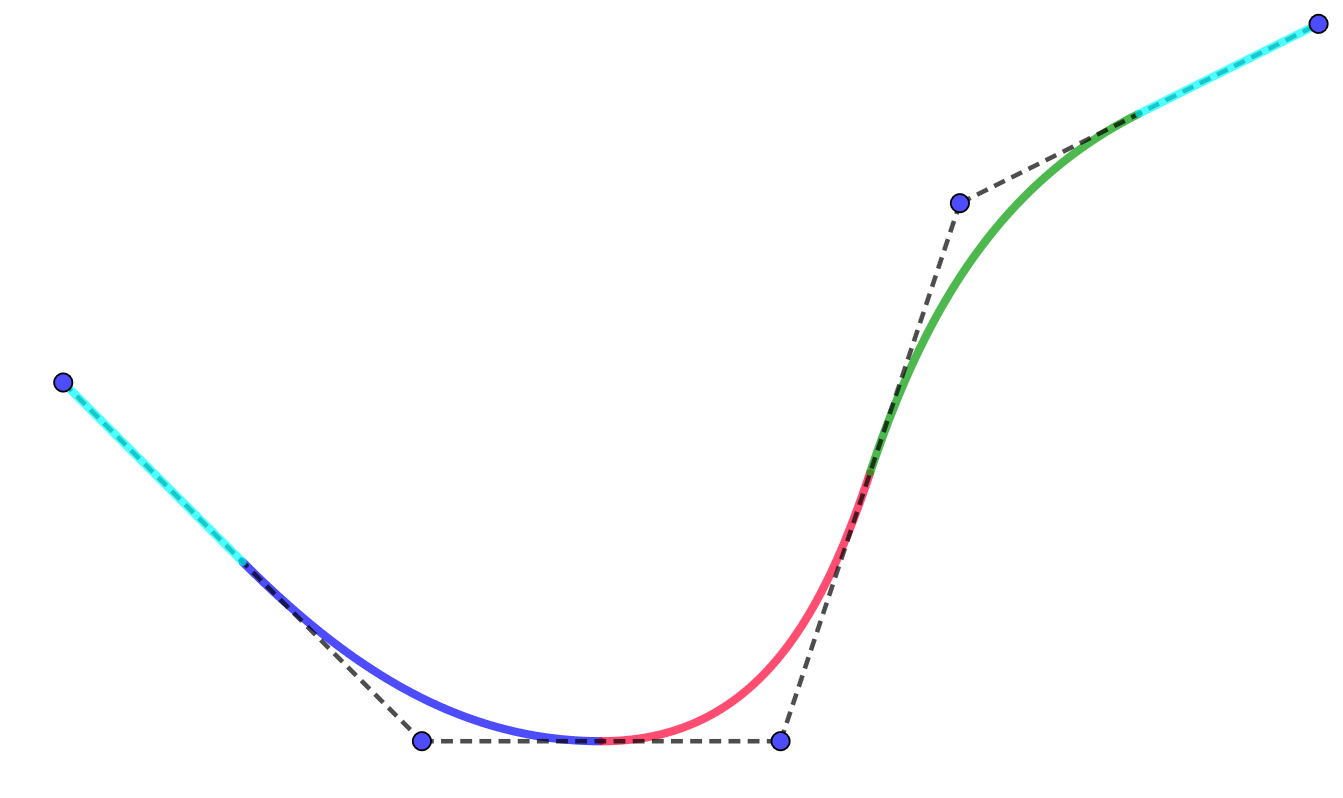}	
\caption{Assigning a spline to a set of points. This figure shows the shape of the curve that we construct.
An ordered set of points $X$ induces a polygonal line. We connect the midpoints of the adjacent segments by cubic functions such that the derivative is continuous and add straight parts on both sides.}
\label{figure:spline}
\end{figure}

We now turn to the construction of an appropriate frame for the discrete spline defined above. As in section \ref{sec:bishp}, once a Bishop frame has been obtained, we may simply define the frame of our given curve as this Bishop frame rotated on each point by a given angle. The task of finding a Bishop frame for our spline curves is deferred to section \ref{sec:bishop_nonarc}.

A framed discrete rod is a pair $ ( X , \Phi ) \in ( \R^{3} )^{N+1} \times \R^{N} $ for some $ N \in \N $.
In order to assign a framed spline to it, we must introduce also a suitable angle function corresponding to $ \Phi $.
Here we simply choose a piecewise linear function. 
More precisely, we define $ \zeta^{X, \Phi} : [0, \ell(X) ] \to \R $ by
\begin{equation*} 
\zeta^{X, \Phi}( \tau ) 
:= \left\{ \begin{array}{cl}
\varphi_{1} , & \tau \le \tau_{1}, \\
\varphi_{i} + \frac{ \varphi_{i+1} - \varphi_{i} }{ \tau_{i+1} - \tau_{i} }  ( \tau - \tau_{i} ), & \tau_{i} \le \tau \le \tau_{i+1} \mbox{ for some } i \in \{ 1 , \ldots, N-1 \}, \\
\varphi_{N}, & \tau_{N} \le \tau .
\end{array} \right.
\end{equation*}
Subsequently, we rescale as above and get
\begin{equation*}
z^{X,\Phi}: [0,L] \to \R, \quad
z^{X,\Phi}(t) := \zeta^{X,\Phi}( \lambda(X) t ). 
\end{equation*}
Hence, we identify a framed discrete rod $ ( X , \Phi ) \in ( \R^{3} )^{N+1} \times \R^{N} $ with a pair of functions
$ ( y^{X} , z^{X, \Phi} ) \in W^{2,2}( (0,L) ; \R^{3} ) \times W^{1,2}((0,L)) $.

\subsection{Bishop frame for non arc-length parametrized curves} \label{sec:bishop_nonarc}

Let us determine the differential equations for a Bishop frame if a curve $ y : [0,L] \to \R^{3} $ is not arc-length parametrized, which is in particular the case for the splines described above. We denote the given parameter by $t$ and the natural by $s$. Let us write
\[ y(t) = u(s(t)) \quad \mbox{and} \quad u(s) = y(t(s)), \]
where we as usual do not use special symbols for the function $ t \mapsto s(t) = \int_{0}^{t} | y'(t) | \ dt $ and its inverse $ s \mapsto t(s) $. Note that $ s'(t) = | y'(t) | $ and $ t'(s) = | y'(t(s)) |^{-1} $.
Then 
\[ u'(s) = \frac{ y'(t(s)) }{ | y'(t(s)) | } \]
and 

\begin{equation*}
u''(s) 
= \frac{ \frac{ y''(t(s)) }{ | y'(t(s)) | } | y'(t(s)) | 
- y'(t(s)) \left( \frac{ y'(t(s)) }{ | y'(t(s)) | } \cdot \frac{ y''(t(s)) }{ | y'(t(s)) | } \right)  }
{ | y'(t(s)) |^{2} }.
\end{equation*} 
Hence,
\begin{equation*}
u'(s(t)) = \frac{ y'(t) }{ | y'(t) | } 
\quad \mbox{and} \quad
u''(s(t)) 
= \frac{ y''(t) 
-  \left( \frac{ y'(t) }{ | y'(t) | } \cdot y''(t) \right) \frac{ y'(t) }{ | y'(t) | } }
{ | y'(t) |^{2} } 
\end{equation*} 
and
\begin{equation*}
u'(s(t)) \times u''(s(t)) 
= \frac{ y'(t) }{ | y'(t) | } \times \frac{ y''(t) 
-  \left( \frac{ y'(t) }{ | y'(t) | } \cdot y''(t) \right) \frac{ y'(t) }{ | y'(t) | } }
{ | y'(t) |^{2} } 
= \frac{ y'(t) \times y''(t) }{ | y'(t) |^{3} }. 
\end{equation*} 
How can we obtain a Bishop frame $ t \mapsto \mathbf{B}(s(t)) =: ( \beta_{1}(t) , \beta_{2}(t) , \beta_{3}(t) ) \in {\rm SO}(3) $
directly? Clearly,
\[ \beta_{1}(t) = \frac{ y'(t) }{ | y'(t) | }. \]
For $ i=2,3 $ 
we may rewrite the system
\[ b_{i}'(s) = ( u'(s) \times u''(s) ) \times b_{i}(s) \]
as
\[ \beta_{i}'(t) 
= b_{i}'(s(t)) | y'(t) |
= \frac{ y'(t) \times y''(t) }{ | y'(t) |^{2} } \times \beta_{i}(t). \]
(If we use only one parametrization in the following, we simply use $ b_{i} $ instead of $ \beta_{i} $.)
We note that the Bishop frame for our piecewise polynomial curves is somewhat special. Since those splines are piecewise planar, it agrees on each planar section with the Frenet frame modulo a constant rotation. 

More importantly, we note that for a converging sequence of curves we may coherently choose Bishop frames so that they converge as well.
\begin{lemma} \label{lem:frame_conv}
Suppose $ y_{N} \weakly y $ in $ W^{2,2}( (0,L) ; \R^{3} ) $ with $ | y' | \ge c > 0 $ a.e.
Then there exist Bishop frames $ \mathbf{B}^{N} $ and $ \mathbf{B} $ on $ y_{N} $ and $y$, resp., such that
$ \mathbf{B}^{N} \weakly \mathbf{B} $ in $ W^{1,2} $ (and therefore strongly in $ C^{0,\alpha} $ for $ \alpha < \frac{1}{2} $).
\end{lemma}

\begin{proof} We adapt the proof of \cite[Lemma 3.4]{hornung2020deformation}.
Let $ 0 \le \alpha < \frac{1}{2} $ be arbitrary.
Regarding the first columns, since $ y_{N} \weakly y $ in $ W^{2,2} $, also $ y_{N} \to y $ in $ C^{1,\alpha} $ and $ y_{N}' \to y' $ in $ C^{0,\alpha} $.
Hence $ b^{N}_{1} = \frac{ y_{N}' }{ | y_{N}' | } $ converges to $ b_{1} = \frac{ y' }{ | y' | } $ weakly in $ W^{1,2} $ and strongly in $ C^{0,\alpha} $. 
The second and third columns must solve the equations
\begin{equation*}
    ( b_{i}^{N} )' = \omega^{N} \times b_{i}^{N}
    \quad \mbox{and} \quad
    b_{i}' = \omega \times b_{i}
\end{equation*} 
with 
\begin{equation*}
    \omega^{N} = \frac{ y_{N}' \times y_{N}'' }{ | y_{N}' |^{2} }
    \quad \mbox{and} \quad
    \omega = \frac{ y' \times y'' }{ | y' |^{2} }.
\end{equation*} 
Since $ b_{1}^{N}(0) \to b_{1}(0) $, we may choose the initial conditions is such way that $ b_{i}^{N}(0) \to b_{i}(0) $ also for $ i = 2,3 $. Thus we defined $ W^{1,2}$-frames $ \mathbf{B}^{N} $ and $ \mathbf{B} $.

Let $ i \in \{2,3\} $ and let us denote $ \delta_{i}^{N} = b_{i}^{N} - b_{i} $. Then
\begin{equation*}
    ( \delta_{i}^{N} )' = ( b_{i}^{N} - b_{i} )' 
    = \omega^N \times ( b_{i}^{N} - b_{i} ) + ( \omega^{N} - \omega ) \times b_{i}
    = \omega^N \times \delta_{i}^{N} + ( \omega^{N} - \omega ) \times b_{i}. 
\end{equation*} 
Set
\begin{equation*}
    \Phi_{i}^{N}(t) = \int_{0}^{t} | \delta_{i}^{N}(\tau) |^2 \ d \tau
    \quad \mbox{and} \quad
    \Psi_{i}^{N}(t) = \left| \int_{0}^{t} ( \omega^{N}( \tau)  - \omega(\tau) ) \times b_{i}( \tau) \ d \tau \right|^2.
\end{equation*} 
Then
\begin{align*}
    ( \Phi_{i}^{N} )'(t) 
    & = | \delta_{i}^{N}(t) |^2 \\
    & = \left| \int_{0}^{t} ( \delta_{i}^{N} )'(\tau) \ d \tau + \delta_{i}^{N}(0) \right|^2 \\
    & \le 2 \left| \int_{0}^{t} \Big( \omega^N (\tau) \times \delta_{i}^{N} (\tau)  + ( \omega^{N}( \tau)  - \omega(\tau) ) \times b_{i}( \tau) \Big) d \tau \right|^2 + 2 | \delta_{i}^{N}(0) |^2 \\
    & \le 4 \left| \int_{0}^{t} \omega^N (\tau) \times \delta_{i}^{N} (\tau) \ d \tau \right|^2
    + 4 \left| \int_{0}^{t}  ( \omega^{N}( \tau)  - \omega(\tau) ) \times b_{i}( \tau) \ d \tau \right|^2 
    + 2 | \delta_{i}^{N}(0) |^2 \\
    & \le 4 \| \omega^{N} \|_{L^2}^2 \cdot \Phi_{i}^{N}(t)
    + 4 \Psi_{i}^{N}(t) + 2 | \delta_{i}^{N}(0) |^2.
\end{align*} 
By Gröwall's inequality, from
\begin{equation*}
    ( \Phi_{i}^{N} )'(t)
    \le a \Phi_{i}^{N}(t) + b ( \Psi_{i}^{N}(t) + | \delta_{i}^{N}(0) |^2 )
\end{equation*}
it follows
\begin{equation*}
    \Phi_{i}^{N}(t)
    \le e^{ a t } \Phi_{i}^{N}(0) + \int_{0}^{t} b e^{ a ( t - \tau ) } ( \Psi_{i}^{N}( \tau ) + | \delta_{i}^{N}(0) |^2 ) \ d \tau
    \le C \left( \int_{0}^{L} \Psi_{i}^{N}( \tau ) \ d \tau + | \delta_{i}^{N}(0) |^{2} \right),
\end{equation*}
employing that $ \Phi_{i}^{N}(0) = 0 $ by the definition. 
Since $ \omega^{N} \weakly \omega $ in $ L^2 $, for every $ t \in [0,L] $ we have
\begin{equation*}
    \lim_{ N \to \i } \Psi_{i}^{N}(t) = 0
    \quad \mbox{and} \quad
    \Psi_{i}^{N}(t) \le \| \omega^{N}  - \omega \|_{ L^2 }^{2} \le C \ \mbox{for all $N$}.
\end{equation*}
Consequently, we may apply the dominated convergence theorem, yielding  
\begin{equation*}
    \lim_{ N \to \i } \int_{0}^{L} \Psi_{i}^{N}( \tau ) \ d \tau = 0.
\end{equation*}
Since by our choice of initial conditions $ \delta_{i}^{N}(0) \to 0 $, we have uniform convergence $ \Phi_{i}^{N} \to 0 $ as $ N \to \i $. 
Hence, $ b_{i}^{N} \to b_{i} $ in $ L^{2} $.
Moreover, from $ | ( b_{i}^{N} )' | \le | \omega^{N} | $, 
it follows that $ \| ( b_{i}^{N} )' \|_{L^2} \le \| \omega^{N} \|_{L^2} \le C $.
Consequenly, $ b_{i}^{N} \weakly b_{i} $ in $ W^{1,2} $, and, finally,  $ \mathbf{B}^{N} \weakly \mathbf{B} $ in $ W^{1,2} $.
\end{proof}


\subsection{The discrete energy and our main result}

The energy of a framed discrete rod $ (X,\Phi) \in ( \R^{3} )^{N+1} \times \R^{N} $ also consists of bending and twisting energy. A-priori, there is no control over the length of this discrete rod. Since in some suitable limit  we want to approximate an arc-length parametrized curve of length $L$, we must introduce a penalty for the deviation of the discrete rod from arc-length parametrization.
Therefore, let
\begin{equation}
\label{eq:discrete-energy}
\F_{N}^{\rm disc}(X,\Phi)
:= \F_{N}^{\rm bend}(X) + \F_{N}^{\rm tor}(X,\Phi) + \F_{N}^{\rm pen}(X). 
\end{equation} 
Since we identified a framed discrete rod with a pair of sufficiently regular functions, 
we may use the same energies as in the continuous case, i.e.,
\begin{equation}
\label{eq:discrete-energy-bending}
\F_{N}^{\rm bend}(X) := \F^{\rm bend}( y^{X} ) = \int_{0}^{L} | ( y^{X} )''(t) |^{2} \ dt 
\end{equation} 
and
\begin{equation*}
\F_{N}^{\rm tor}(X,\Phi) := \F^{\rm tor}( z^{X,\Phi} ) = \int_{0}^{L} | ( z^{X,\Phi} )'(t) |^{2} \ dt.
\end{equation*} 
%
%
For the penalty term we choose 
\begin{equation}
\label{eq:penalty-term}
\F_{N}^{\rm pen}(X) = N^{\alpha} | \lambda(X) - 1 | + N^{\beta} \max_{i} | x_{i} - x_{i-1} |
\end{equation}
for some $ \alpha \in (0,2) $ and $ \beta \in (0,1) $.

\begin{remark}
As mentioned above, the penalty term will ensure that we approximate an arc-length parametrized curve. The choice and its role will become clear in the proofs of the $\liminf$-inequality and of the existence of a recovery sequence, see Remark~\ref{remark:penalty}, Example~\ref{ex:spacing} and the conclusion in Subsection~\ref{subsect:conclusion}.

We note that one may also use a hard penalty of the form
\[
\F_{N}^{\rm pen}(X) = \begin{cases}0 & \text{if $| \lambda(X) - 1 | < N^{-\alpha}$ and $\max_{i} | x_{i} - x_{i-1} | \le N^{-\beta}$ }, \\
+\infty & \text{otherwise,}\end{cases}
\]
for $\alpha$, $\beta$ as above. In a gradient-flow algorithm, this may be enforced for example by restricting change of the node distances, similar as, e.g., done in \cite{bartels2013simple} when the initial discrete curve has approximately the correct length.
\end{remark}

Now we state the main results of this work, i.e., compactness and $ \Gamma $-convergence of the discrete energies.
We must first have functionals defined on the same space.
For that purpose, let us for every $ X \in ( \R^{3} )^{N+1} $ and $ \Phi \in \R^{N} $ define
\begin{equation}
\label{eq:def-F-N}
    \F_{N}(y^{X},z^{X, \Phi}) := \F^{\rm disc}_{N}( X, \Phi ),
\end{equation} 
and extend the definition to the whole $ L^{2}((0,L);\R^{3}) \times L^{2}((0,L)) $ by $ \i $.
(Since $ \F^{\rm disc}_{N} $ depends only upon the assigned functions, the functional $ \F_{N} $ is well-defined.)

Trivially the following holds.
\begin{proposition}[Equicoercitivity]
Let $ \F_{N}: L^{2}((0,L);\R^{3}) \times L^{2}((0,L)) \to [0,\i] $ be defined as in \eqref{eq:def-F-N}.
Then for any sequence $ (y_{N},z_{N}) \in W^{2,2}((0,L);\R^{3}) \times W^{1,2}((0,L))$
\[ \F_{N}(y_{N},z_{N}) \ge  \| y_{N}'' \|_{L^{2}}^{2} +  \| z_{N}' \|_{L^{2}}^{2}. \]
\end{proposition}

The above noted equicoercitivity, together with Lemma \ref{lem:frame_conv}, immediately implies the following.
\begin{corollary}[Convergence of frames]
\label{theo:frames}
Suppose that  
\[ \F_N(y_N,z_N) \le C
\quad \mbox{and} \quad
(y_N,z_N) \to (y,z) \quad \mbox{in } L^{2}((0,L);\R^{3}) \times L^{2}((0,L)). \]
Let us choose any frames $ \mathbf{R}_{N} , \mathbf{R} \in W^{1,2}((0,L);\SO(3)) $ that are coherent with $ (y_N,z_N) $ and $ (y,z) $, respectively, i.e., some Bishop frames determined by $y_{N}$ and $y$ rotated around the tangent by the angles $z_{N}$ and $z$, respectively, see \eqref{eq:arbitrary-frame}. 

Then, there exist $\bar\theta_N \in \R$ so that  $ \mathbf{R}_{N} \mathbf{\Theta}^{\bar\theta_N} \to \mathbf{R} $ in $ C^{0,\alpha} $.
In other words, the frames converge in $ C^{0,\alpha} $ modulo rotation by constant angles.
\end{corollary}



\begin{theorem}[$\Gamma$-convergence]
\label{theo:main}
Let $ \F_{N}, \F : L^{2}((0,L);\R^{3}) \times L^{2}((0,L)) \to [0,\i] $ be defined as in \eqref{eq:def-F-N} and \eqref{eq:continuous-energy} respectively.
Then 
\[ \Gamma\text{-}\lim_{ N \to \i } \F_{N} = \F \]
in $ L^{2}((0,L);\R^{3}) \times L^{2}((0,L)) $ (in the norm topology).
\end{theorem}

The proofs of these theorems are organized as follows. In Section~\ref{sect:lowerbound} we prove the $\liminf$-inequality and also partially justify the choice of the penalty terms from~\eqref{eq:penalty-term}.
In order to find a recovery sequence, we first in Section~\ref{sect:app} approximate a smooth curve with polygonal lines whose segments have equal lengths. The corresponding splines are then even piecewise quadratic. Thus, their  bending energies from \eqref{eq:discrete-energy-bending} are a finite sums. We use such splines in Section~\ref{sect:recovery} for the construction of a recovery sequence.

Some remarks are in order.
\begin{remark}
Corollary \ref{theo:frames} ensures that the $\Gamma$-convergence result is compatible with physically reasonable boundary conditions: at the clamped end (or ends) of a rod one wants to fix a \emph{frame}, and not the rotation $\theta$ compared to a Bishop frame. Since the frames of approximating curves converge strongly with respect to a H{\"o}lder-norm, however, the set of curves satisfying a clamped boundary condition is closed. The converse, namely the construction of a recovery sequence satisfying a clamped boundary condition is also possible, see Remark \ref{rem:clamped_frame_recovery}.
\end{remark}

\begin{remark}
We identified the discrete description of our rods as sets of points and angles with framed spline curves to simplify an $\Gamma$-convergence approach. This identification yields coinciding function spaces for discrete and limiting continuous curves. On the other hand, it would be more natural for discrete data to write the energy as a finite sum with each term being uniquely determined by a triple of consecutive points and corresponding two angles. We elaborate on this in section \ref{sec:triple} just below.
\end{remark}

\subsection{An energy depending only on nearest neighbors} \label{sec:triple}

Let us for a triple of points $ x_{-1} , x_{0} , x_{1} \in \R^{3} $ define the corresponding bending energy by
\begin{equation*}
\mathrm{Bend}( x_{-1}, x_{0} , x_{1} )
:= 2 \sin^2 \frac{ \phi }{2} \cdot \frac{ r_{1}^{3} + r_{0}^{3} }{ ( \frac{ r_{1} + r_{0} }{2} )^{4} }
\end{equation*}
with $ r_{0} := | x_{0} - x_{-1} | $, $ r_{1} := | x_{1} - x_{0} | $ and $ \phi $ the angle between $ x_{1} - x_{0} $ and $ x_{0} - x_{-1} $, see Figure~\ref{figure:triple}.
%
%
%
If $ r_{0} = r_{1} = r $, this expression reduces to
\begin{equation*}
\mathrm{Bend}( x_{-1}, x_{0} , x_{1} )
= \frac{4}{r} \sin^2 \frac{ \phi }{2} 
\end{equation*}
This is closely related to (and in our limit of approximating a $W^{2,2}$-curve agrees with, since necessarily $\phi\to 0$ as $r\to 0$) the bending energy in \cite{bergou2008discrete}, as they have $ \frac{4}{r} \tan^2 \frac{ \phi }{2}  $ (if we set $ \beta = 2 $ in their expression for the bending energy).

\begin{figure}[ht]
\includegraphics[width=.57\textwidth]{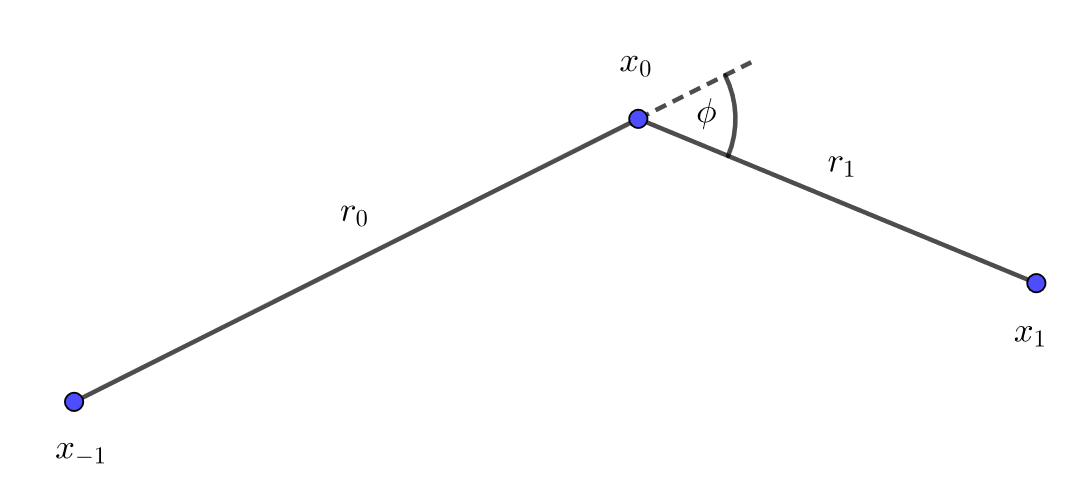}	
\caption{For a discrete rod, a more natural concept of a bending energy is the sum of energies in each inner node.
Since the local form around a node is given by the lengths of the adjacent edges and on the angle between them, the energy should depend only on them.}
\label{figure:triple}
\end{figure}

Moreover, if also the twisting angles $ \varphi_{0} , \varphi_{1} \in \R $ are given, let

\begin{equation*}
\mathrm{Tor}( x_{-1}, x_{0} , x_{1} , \varphi_{0} , \varphi_{1} )
:=  \frac{ | \varphi_{1} - \varphi_{0} |^{2} }{ \frac{ r_{1} + r_{0} }{2} }.
\end{equation*}

\begin{proposition}
Let us for any discrete framed rod
$ (X,\Phi) \in ( \R^{3} )^{N+1} \times \R^{N} $
define its local bending and torsional energy by
\begin{equation*}
\F_{N}^{\rm bend, loc}(X)
:= \sum_{i=1}^{N-1} \mathrm{Bend}( x_{i-1}, x_{i} , x_{i+1} )
\end{equation*}
and
\begin{equation*}
\F_{N}^{\rm tor, loc}(X,\Phi)
:= \sum_{i=1}^{N-1} \mathrm{Tor}( x_{i-1}, x_{i} , x_{i+1} , \varphi_{i-1} , \varphi_{i} ).
\end{equation*}
Then
\begin{equation*}
\F_{N}^{\rm bend}(X) = \lambda(X)^{3} \F_{N}^{\rm bend, loc}(X)
\quad \mbox{and} \quad
\F_{N}^{\rm tor}(X,\Phi) = \lambda(X) \F_{N}^{\rm tor, loc}(X,\Phi).
\end{equation*}
Consequently, if we replace in the definition of $ \F_{N}^{\rm disc} $ in \eqref{eq:discrete-energy} the functionals 
$ \F_{N}^{\rm bend} $ and $ \F_{N}^{\rm tor} $ by the functionals $ \F_{N}^{\rm bend,loc} $ and $ \F_{N}^{\rm tor,loc} $, the $ \Gamma $-convergence result in Theorem~\ref{theo:main} still holds.
\end{proposition} 

\begin{proof}
We use the denotations from subsection \ref{subsect:assign-spline}.
From
\[ (y^{X})''(t)
= \frac{d^2}{\t^2 } \eta^{X}( \lambda(X) t ) 
= \lambda(X)^{2} (\eta^{X})''( \lambda(X) t ), \]
it follows
\[ \F_{N}^{\rm bend}(X)
= \int_{0}^{L} | ( y^{X} )''(t) |^{2} \t 
= \int_{0}^{L} \lambda(X)^{4}  | (\eta^{X})''( \lambda(X) t ) |^{2} \t 
= \int_{0}^{ \ell(X) } \lambda(X)^{3}  | (\eta^{X})''( \tau ) |^{2} \ d\tau. \]
The function $ \eta^{X} $ was constructed in \eqref{eq:def-eta} from cubic functions
\[ S_{i}( \tau ) = A_{i} \tau^{3} + B_{i} \tau^{2} + \frac{ x_{i} - x_{i-1} }{ | x_{i} - x_{i-1} | } \tau + \frac{1}{2} ( x_{i-1} + x_{i} ) \]
and two straight parts. Hence
\[ \int_{0}^{ \ell(X) } | (\eta^{X})''( \tau ) |^{2} \ d\tau 
= \sum_{i=1}^{N-1} \int_{0}^{ \tau_{i+1} - \tau_{i} } | S_{i}''( \tau ) |^{2} \ d\tau 
= \sum_{i=1}^{N-1} \int_{0}^{ T_{i} } | 6 A_{i} \tau + 2 B_{i} |^{2} \ d\tau \]
with $ T_{i} = \frac{1}{2} ( | x_{i+1} - x_{i} | + | x_{i} - x_{i-1} | ) $.

By computing scalar products of the expressions
\begin{eqnarray*}
B_{i} T_{i}^{2} 
&  =  & \left( 1 - \tfrac{ | x_{i} - x_{i-1} | }{ 2 | x_{i+1} - x_{i} | } \right) ( x_{i+1} - x_{i} ) 
        + \left( \tfrac{1}{2} - \tfrac{ | x_{i+1} - x_{i} | }{ | x_{i} - x_{i-1} | } \right) ( x_{i} - x_{i-1} ), \\
A_{i} T_{i}^{3} 
&  =  & \left( - \tfrac{1}{2} + \tfrac{ | x_{i} - x_{i-1} | }{ 2 | x_{i+1} - x_{i} | } \right) ( x_{i+1} - x_{i} ) 
		+ \left( - \tfrac{1}{2} + \tfrac{ | x_{i+1} - x_{i} | }{ 2 | x_{i} - x_{i-1} | } \right) ( x_{i} - x_{i-1} ),
\end{eqnarray*}
%
%
%
for $ r_{i} := | x_{i} - x_{i-1} | $ and $ \phi_{i} $ being the angle between $ x_{i+1} - x_{i} $ and $ x_{i} - x_{i-1} $, we arrive at
\begin{eqnarray*}
| A_{i} |^{2} T_{i}^{6} &  =  & 
\sin^{2} \tfrac{ \phi_{i} }{2} \cdot ( r_{i+1} - r_{i} )^2, \\
( A_{i} \cdot B_{i} ) T_{i}^{5} &  =  & 
- \sin^{2} \tfrac{ \phi_{i} }{2} \cdot ( 2 r_{i+1} - r_{i} ) ( r_{i+1} - r_{i} ), \\
| B_{i} |^{2} T_{i}^{4} &  =  & 
\sin^{2} \tfrac{ \phi_{i} }{2} \cdot ( 2 r_{i+1} - r_{i} )^2.
\end{eqnarray*}
This yields
\begin{align*}
& \int_{0}^{ T_{i} } | 6 A_{i} \tau + 2 B_{i} |^{2} \ d\tau \\
&  =  12 | A_{i} |^2 T_{i}^{3} + 12 ( A_{i} \cdot B_{i} ) T_{i}^{2} + 4 | B_{i} |^{2} T_{i} \\
&  =  \frac{ 12 \sin^{2} \tfrac{ \phi_{i} }{2} \cdot ( r_{i+1} - r_{i} )^2 
        - 12 \sin^{2} \tfrac{ \phi_{i} }{2} \cdot ( 2 r_{i+1} - r_{i} ) ( r_{i+1} - r_{i} )
        + 4 \sin^{2} \tfrac{ \phi_{i} }{2} \cdot ( 2 r_{i+1} - r_{i} )^2 }{ T_{i}^{3} } \\
&  =  2 \sin^{2} \tfrac{ \phi_{i} }{2} \cdot \frac{ r_{i+1}^{3} + r_{i}^{3} }{ ( \frac{ r_{i+1} + r_{i} }{2} )^{4} } \\
&  =  \mathrm{Bend}( x_{i-1}, x_{i} , x_{i+1} ).
\end{align*}
Thus, the expression for the bending energies is proven.

Analogously, we have
\[ \int_{0}^{L} | ( z^{X,\Phi} )'(t) |^{2} \ dt
= \int_{0}^{ \ell(X) } \lambda(X) | ( \zeta^{X,\Phi} )'( \tau ) |^{2} \ d \tau \]
and
\[ \int_{0}^{ \ell(X) }  | ( \zeta^{X,\Phi} )'( \tau ) |^{2} \ d \tau 
= \sum_{i=0}^{N} \int_{ \tau_{i} }^{ \tau_{i+1} }  | ( \zeta^{X,\Phi} )'( \tau ) |^{2} \ d \tau 
= \sum_{i=1}^{N-1} \left( \frac{ \varphi_{i+1} - \varphi_{i} }{ \tau_{i+1} - \tau_{i} } \right)^2 ( \tau_{i+1} - \tau_{i} )
= \sum_{i=1}^{N-1} \frac{ | \varphi_{i+1} - \varphi_{i} |^{2} }{ T_{i} }.  \]
The comment about $ \Gamma $-convergence follows from the fact that the penalty term in $ \F_{N}^{\rm disc} $ insures that $ \lambda(X) \to 1 $.
\end{proof}

\section{Lower bound and compactness} \label{sect:lowerbound}
We start by showing the $ \liminf $-inequality in Theorem~\ref{theo:main}.
Let $ ( y_{N} , z_{N} ) $ be a sequence in $ L^{2}( (0,L) ; \R^{3} ) \times L^{2}((0,L)) $ that converges to $ (u,\theta) $ (in the norm topology).
As usual, we may restrict ourselves to the sequences whose members lie in the domains of the corresponding functionals.
Hence, for every $ N \in \N $ there exist
\[ X_{N} = ( x^{N}_{0} , \ldots , x^{N}_{N} ) \in (\R^{3})^{N+1}
\quad \mbox{and} \quad 
\Phi_{N} = ( \varphi^{N}_{1} , \ldots , \varphi^{N}_{N} ) \in \R^{N} \]
such that $ y_{N} = y^{X_{N}} $ and $ z_{N} = z^{X_{N},\Phi_{N}} $.
Since
\[ \F_{N}( y_{N} , z_{N} ) = \F^{\rm bend}( y_{N} ) + \F^{\rm tor}( y_{N} , z_{N} ) + \F_{N}^{\rm pen}( X_{N} ), \]
%
we may moreover suppose that the sequences $ \F^{\rm bend}( y_{N} ) $, $ \F^{\rm tor}( y_{N} , z_{N} ) $ and $ \F_{N}^{\rm pen}( X_{N} ) $ even converge (to finite values).
Hence,
\begin{equation*}
\| y_{N}'' \|_{ L^{2} } \le C. 
\end{equation*}
From Poincar\'{e}'s inequality it follows that $ y_{N} $ is bounded in $ W^{2,2}( (0,L) ; \R^{3} ) $.
By extracting a non-relabeled subsequence, we get $ y_{N} \weakly u $ in $ W^{2,2}( (0,L) ; \R^{3} ) $ and thus $ y_{N} \to u $ in $ C^{1,\alpha}( [0,L] ; \R^{3} ) $ for any $ \alpha \in [0,1/2) $.
Since $ y_{N}'' \weakly u'' $ in $ L^{2} $, by lower semi-continuity of the norm in the weak topology we have
\begin{equation}
\label{eq:liminf-bend-1}
\lim_{N \to \i} \F^{\rm bend}( y_{N} ) 
= \lim_{N \to \i} \| y_{N}'' \|_{ L^{2} }^{2} 
\ge \| u'' \|_{ L^{2} }^{2}.
\end{equation}
Let us denote $ \lambda_{N} := \lambda( X_{N} ) $.
Let $ \tau^{N}_{i} $ be defined for $ X_{N} $ as in \eqref{eq:def-tau}.
Then $ t^{N}_{i} := \lambda_{N}^{-1} \tau^{N}_{i} $ are the points of possible non-continuity of $ y''_{N} $.
As we saw in \eqref{eq:properties-spline-knots},
\[ | y_{N}'( t^{N}_{i} ) | = \lambda_{N}. \]
For every $ t^{N}_{i} \le t \le t^{N}_{i+1} $, $ i \in \{ 1 , \ldots , N-1 \} $, we notice that 
\[ y_{N}'( t ) - y_{N}'( t^{N}_{i} )
= \int_{ t^{N}_{i} }^{ t } y_{N}''( \xi ) \ d \xi.
\]
Thus, 
\[ | | y_{N}'( t ) | - \lambda_{N} |
\le \int_{ t^{N}_{i} }^{ t } | y_{N}''( \xi ) | \ d \xi
\le C \sqrt{ t^{N}_{i+1} - t^{N}_{i} }
= \frac{C}{ \lambda_{N}^{1/2} } \max_{i} \sqrt{ | x^{N}_{i+1} - x^{N}_{i} | }
\]
since 
\[ t^{N}_{i+1} - t^{N}_{i} = \frac{ \tau^{N}_{i+1} - \tau^{N}_{i} }{ \lambda_{N} } = \frac{ | x^{N}_{i+1} - x^{N}_{i} | }{ \lambda_{N} }. \]
Therefore,
\begin{equation*}
| | y_{N}'(t) | - 1 |  
\le | | y_{N}'(t) | - \lambda_{N} |  + | \lambda_{N} - 1 |
\le \frac{C}{ \lambda_{N}^{1/2} } \max_{i} \sqrt{ | x^{N}_{i+1} - x^{N}_{i} | } + | \lambda_{N} - 1 |.
\end{equation*}
For $t \le t_{1}$ and $ t \ge t_{N} $, we have $ | y_{N}'(t) | = \lambda_{N} $, see \eqref{eq:properties-spline-linear}.
From $ \F^{\rm pen}_{N}( X_{N} ) \le C $ it follows that $ \lambda_{N} \to 1 $ and $ \max_{i} | x^{N}_{i+1} - x^{N}_{i} | \to 0 $.
Hence, for every $ t \in [0,L] $
\[ 1 = \lim_{ N \to \i } | y_{N}'(t) | =  | u'(t) |. \]
Therefore, $ u $ lies in the domain of $ \F^{\rm bend} $ and $ \F^{\rm bend}(u) = \| u'' \|_{L^{2}} $. 
This together with \eqref{eq:liminf-bend-1} implies
\begin{equation}
\label{eq:liminf-bend-2}
\lim_{N \to \i} \F^{\rm bend}( y_{N} ) 
\ge \F^{\rm bend}( u ).
\end{equation}
In a similar manner, we furthermore get a non-relabeled subsequence $ z_{N} $ such that $ z_{N} \weakly \theta $ in $ W^{1,2}((0,L)) $ and thus $ z_{N} \to \theta $ in $ C^{0,\alpha}([0,L]) $ for any $ \alpha \in [0,1/2) $.
Again we employ lower semi-continuity of the norm in the weak topology to conclude that $ z_{N}' \weakly \theta' $ in $ L^{2} $ implies
\begin{equation}
\label{eq:liminf-tor} 
\lim_{N \to \i} \F^{\rm tor}( z_{N} ) 
= \lim_{N \to \i} \| z_{N}' \|_{ L^{2} }^{2}
\ge \| \theta' \|_{ L^{2} }^{2} 
= \F^{\rm tor}( \theta). 
\end{equation}
Let us summarize. 
If $ ( y_{N} , z_{N} ) \to (u,\theta) $ in $ L^{2}( (0,L) ; \R^{3} ) \times L^{2}((0,L)) $ 
and $$ \liminf_{N \to \i} \F_{N}( y_{N} , z_{N} ) < \i, $$ then $ (u, \theta ) $ lies in the domain of $ \F $, i.e.
\[ u \in W^{2,2}( (0,L) ; \R^{3} ) \ \mbox{with} \ | u' | \equiv 1
\quad \mbox{and} \quad
\theta \in W^{1,2}( (0,L) ) \]
and by \eqref{eq:liminf-bend-2}, \eqref{eq:liminf-tor}
\[ \liminf_{N \to \i} \F_{N}( y_{N} , z_{N} ) \ge \F( u , \theta ). \]

\begin{remark} \label{remark:penalty}
Along the lines of the proof, we perceive why the penalty term of this form is necessary. 
Since we want to approximate rods by discrete ones, it is clear that we must impose $ \lambda_{N} \to 1 $. 
In our definition of the assigned spline, we rescale the interval and have thus a priori no control on the lengths.
For that reason, we must indeed penalize discrete rods whose length differ excessively from $L$.
Example \ref{ex:spacing} below shows that also the second penalization is necessary. 
Even if $ \lambda_{N} \to 1 $, we may still run in trouble if the fineness of the partitions does not converge to 0.
\end{remark}

\begin{example} \label{ex:spacing}
Let $ L = 3 $, and let us for each $ N \ge 2 $ define the $ (N+1) $-tuple
\[ X_{N} := ( x_{0}^{N} , \ldots , x_{N}^{N} ) \]
with
\begin{equation*}
x_{0}^{N} := (0,0), \quad
x_{i}^{N} := ( \tfrac{4}{2^{N-i}} , 0 ) \ \mbox{for $ i = 1 , \ldots , N-1 $}, \quad
x_{N}^{N} := (2,1). 
\end{equation*}
The sequence of the corresponding splines $ y_{N} := y^{X_{N}} $ is constant for $ N \ge 3 $ and is equal to
\[ y : [0,3] \to \R^{2},
\quad
y(t) :=
\begin{cases}
(t,0) & t \le \frac{3}{2} \\
(t,0) + \frac{1}{2} (t - \frac{3}{2} )^{2} (-1,1)  & \frac{3}{2} \le t \le \frac{5}{2} \\
(2,t-2) & t \ge \frac{5}{2} \\
\end{cases} \]
Hence, $ \lambda_{N} = 1 $ for all $ N \in \N $. For an illustration of the spline, see Figure \ref{fig:spacing}. 
\begin{figure}[ht]
\includegraphics[width=.5\textwidth]{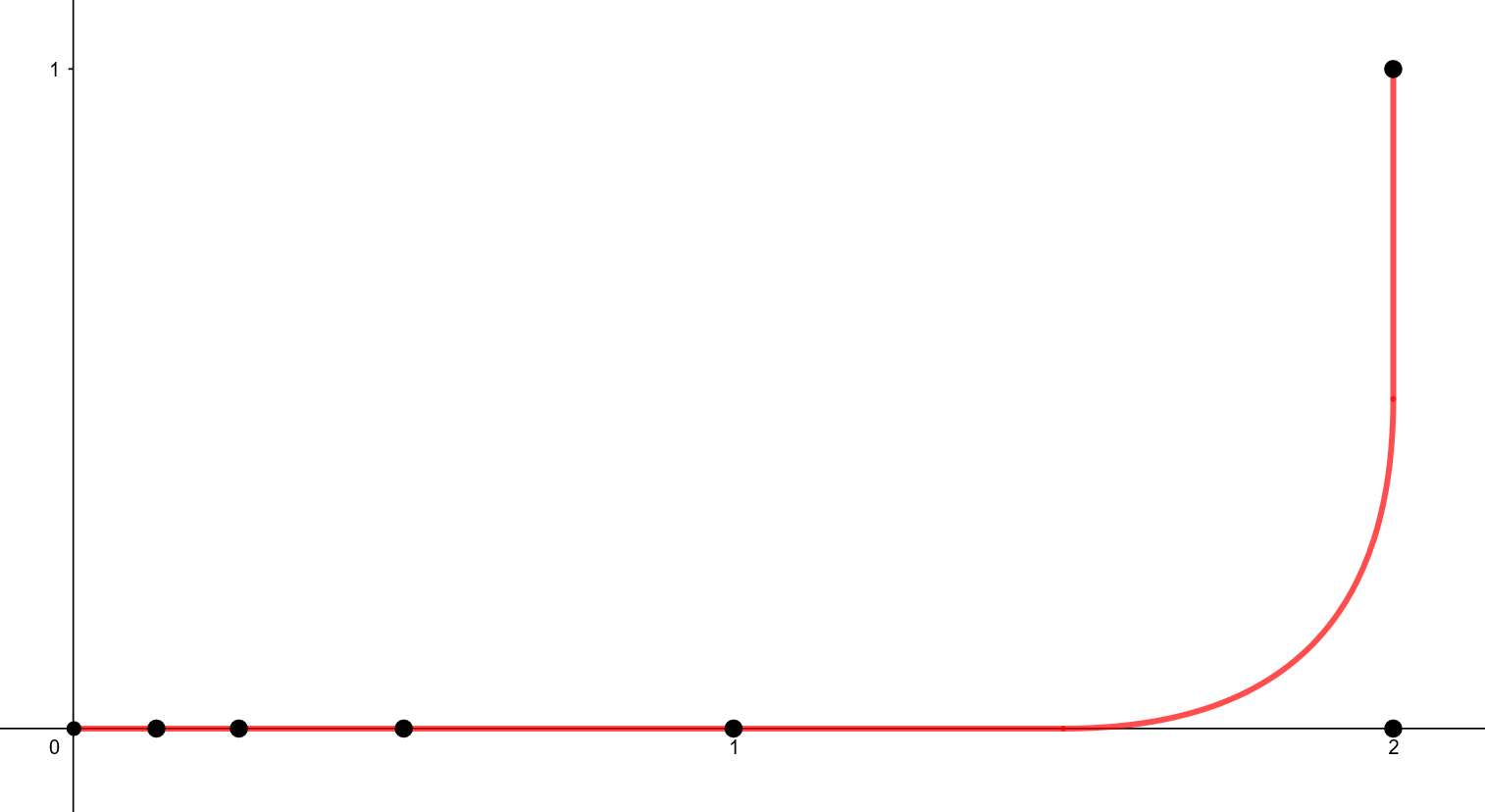}	
\caption{$ X_{N} $ and graph of the spline. In Example~\ref{ex:spacing} we construct ordered set of points in the following way. We start by $(0,0)$, $(1,0)$, $(2,0)$ and $(2,1)$, and bisect in each step the first (the most left) segment. The induced spline, however, remains the same. Here, the situation for $N=6$ is depicted.} 
\label{fig:spacing}
\end{figure}
However, the modulus of
\[ y'(t) =
\begin{cases}
(1,0) & t \le \frac{3}{2} \\
(1,0) + (t - \frac{3}{2} ) (-1,1)  & \frac{3}{2} \le t \le \frac{5}{2} \\
(0,1) & t \ge \frac{5}{2} \\
\end{cases} \]
is not 1 everywhere, e.g., $ y'(2) = ( \frac{1}{2} ,  \frac{1}{2} ) $. 
\end{example}


%
%

%

\section{Approximation of a smooth arc-length parametrized curve by a polygonal line}
\label{sect:app}
Let us suppose that we have an arc-length parametrized smooth curve $ u : [0,L] \to \R^{n} $. To obtain a suiteable discretization, for given $N$, we would like to find $N+1$ points on the curve with equal distance \emph{in ambient space}. Optimally, the first and last point of the discretization should correspond to $u(0)$ and $u(L)$, respectively. In the following we show that such a choice of points can always be found.

\subsection{Relation between the arc length $l$ and the line segment $r$}
\label{subsect:relation-l-r}
First let us take any $ s_{0} \in [0,L] $ and explore the difference between the distance along the curve $l$ and the Euclidean distance $r$ from $ u( s_{0} ) $ to other points on the curve, see Figure~\ref{fig:relation-l-r}.
\begin{figure}[ht!]
\includegraphics[width=\textwidth]{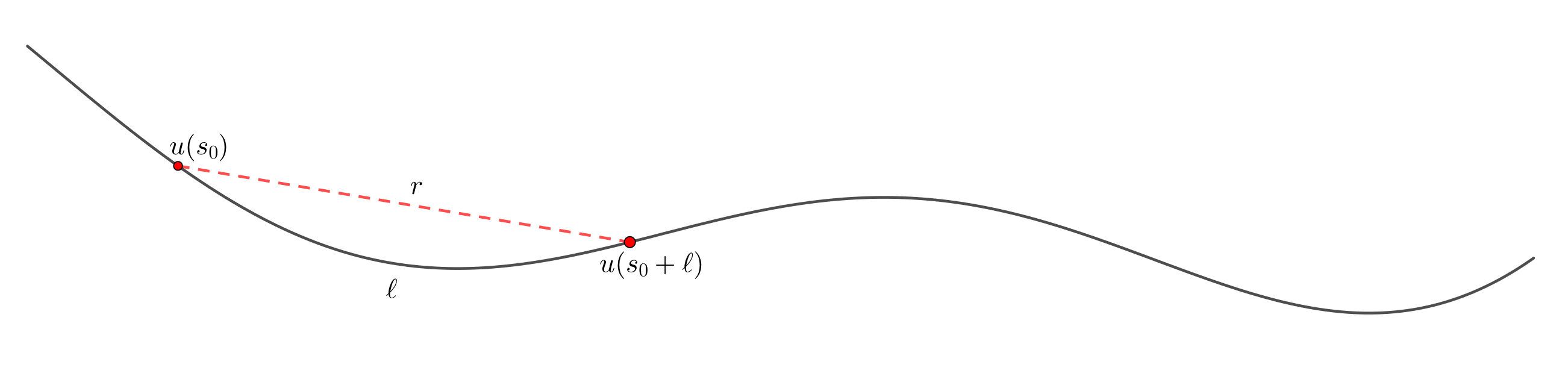}	
\caption{Relation between $l$ and $r$. For our computations, quite fine comparisons of distances between two points on the given arc-length parametrized curve will be needed. $l$ denotes the distance along the curve, while $r$ stands for the Euclidean distance.}
\label{fig:relation-l-r}
\end{figure}

Since the point $ u( s_{0} + l ) $ is $l$ away along the curve,
its Euclidean distance to $ u( s_{0} ) $ is given by
\[ r = | u( s_{0} + l ) - u( s_{0} ) | = \left| \int_{ s_{0} }^{ l + s_{0} }  u'(s) \s \right|. \]
Let us determine the relation $ r = r( s_{0} , l ) $ (and, if possible, $ l = l( s_{0} , r ) $).
We may write
\begin{eqnarray*}
u'(s) 
& = & u'( s_{0} ) + \int_{ s_{0} }^{s} u''( \sigma ) \ d \sigma \\
& = & u'( s_{0} ) + \int_{ s_{0} }^{s} \left( u''( s_{0} ) + \int_{ s_{0} }^{ \sigma } u'''( \tau ) \ d \tau \right) \ d \sigma \\
& = & u'( s_{0} ) + ( s - s_{0} ) u''( s_{0} ) + \int_{ s_{0}}^{s} \int_{ s_{0} }^{ \sigma } u'''( \tau ) \ d \tau \ d \sigma .
\end{eqnarray*}
Since $u$ is arc-length parametrized, we have $ | u'( s_{0} ) | = 1 $ and $ u'( s_{0} ) \perp u''( s_{0} ) $. Therefore,
\begin{eqnarray*}
r( s_{0} , l )
& = & \left| \int_{ s_{0} }^{ l + s_{0} }  u'(s) \s \right| \\
& \ge & \left| u'( s_{0} ) \cdot \int_{ s_{0} }^{ l + s_{0} }  u'(s) \s \right| \\
& =  & \left| \int_{ s_{0} }^{ l + s_{0} } \left( 1 + \int_{ s_{0} }^{s} \int_{ s_{0} }^{ \sigma } u'( s_{0}) \cdot  u'''( \tau ) \ d \tau \ d \sigma \right) \s \right| \\
& =  & \left| l + \int_{ s_{0} }^{ l + s_{0} } \int_{ s_{0} }^{s} \int_{ s_{0} }^{ \sigma }  u'( s_{0}) \cdot u'''( \tau ) \ d \tau \ d \sigma \s \right| \\
& \ge & l - \frac{ l^{3} }{6} \| u''' \|_{\i}.
\end{eqnarray*}

Furthermore, it holds
\begin{eqnarray*}
\partial_{l} r( s_{0} , l )
& = & \frac{ \partial }{ \partial l } \left| \int_{ s_{0} }^{ l + s_{0} }  u'(s) \s \right| \\
& = & \frac{1}{ r( s_{0} , l ) } \int_{ s_{0} }^{ l + s_{0} }  u'(s) \s \cdot u'(l + s_{0})  \\
& = & \frac{1}{ r( s_{0} , l ) } \int_{ s_{0} }^{ l + s_{0} }  \left( u'(l + s_{0}) - \int_{s}^{ l + s_{0} } u''(\sigma) \ d \sigma \right) \s \cdot u'(l + s_{0})  \\
& = & \frac{1}{ r( s_{0} , l ) }  \left( l - \int_{ s_{0} }^{ l + s_{0} } \int_{s}^{ l + s_{0} } u''(\sigma) \ d \sigma \s \cdot u'(l + s_{0}) \right) \\
& \ge & \frac{1}{l} \left(  l -  \frac{ l^{2} }{2} \| u'' \|_{\i} \right) \\
& \ge & 1 -  \frac{l}{2}  \| u'' \|_{\i}.
\end{eqnarray*}
Hence,
\begin{equation}
\label{eq:ass-r-l-s0-1}
l \ge r( s_{0} , l ) \ge l - \frac{ l^{3} }{6} \| u''' \|_{\i} 
\quad \mbox{and} \quad
1 \ge \partial_{l} r( s_{0} , l ) \ge 1 -  \frac{l}{2}  \| u'' \|_{\i}.
\end{equation}  

Note that both assessments are independent of $ s_{0} $.
For sufficiently small $l$ the function $ l \mapsto r(s_{0},l) $ is strictly increasing.
Hence, there exists the function $ r \mapsto l(s_{0},r) $, which is also strictly increasing.

Moreover, there is a constant $C$, dependant on $u$ but not on $ s_{0} $, such that
\begin{equation*}
r(s_{0},l) \ge l - C l^{3} 
\quad \mbox{and} \quad
l(s_{0},r) \le r + C r^{3}
\end{equation*} 
for $ r,l $ sufficiently small.

From 
\begin{eqnarray*}
\partial_{ s_{0} } r( s_{0} , l )
& = & \frac{ \partial }{ \partial s_{0} } \left| u( l + s_{0} ) - u(s_{0}) \right| \\
& = & \frac{ u( l + s_{0} ) - u(s_{0}) }{ r( s_{0} , l ) } \cdot \big( u'( l + s_{0} ) - u'(s_{0}) \big)  \\
& = & \frac{ u( l + s_{0} ) - u(s_{0}) }{ r( s_{0} , l ) } \cdot \int_{s_{0}}^{l + s_{0}} u''( \sigma ) \ d \sigma,
\end{eqnarray*}
it follows
\begin{equation}
\label{eq:ass-r-l-s0-3}
| \partial_{ s_{0} } r( s_{0} , l ) |
\le l \| u'' \|_{\i}. 
\end{equation} 
Since
\[ \frac{ \partial l}{ \partial s_{0} } 
= - \frac{ \partial r}{ \partial s_{0} } \frac{ \partial l}{ \partial r } \]
it follows from \eqref{eq:ass-r-l-s0-1} and \eqref{eq:ass-r-l-s0-3} that for small $r$
%
\begin{equation*}
\left| \frac{ \partial l}{ \partial s_{0} } \right|
= \left| \frac{ \partial r}{ \partial s_{0} } \right| \cdot \left| \frac{ \partial l}{ \partial r } \right| < 1.
\end{equation*}
Therefore, for a fixed small $r$ the function
\begin{equation}
\label{eq:preserving-order}
s_{0} \mapsto s_{0} + l( s_{0} , r ) 
\end{equation} 
is strictly increasing. 
This means that, if we take two points on the curve and find for each of them the next one along the curve that is at the distance $r$, the order is preserved. 

\subsection{Approximation by a polygonal curve}
\label{subsect:app-equi}

Let $ r > 0 $ be arbitrary. Let $ N(r) \in \N \cup \{0\} $ be the largest number of line segments of the length $r$ by which we approximate the curve $u$ starting at $ u(0) $.
More precisely, we set $ s_{0} := 0 $ and $ x_{0} := u(0) $. Then we repeat for $ i = 0,1,2,\ldots $
\begin{itemize} 
\item
If $ | u(s) - u( s_{i} ) | < r $ for all $ s \in [ s_{i} , L ] $, we stop and set $ N(r) := i $.
\item
Otherwise, we set $ s_{i+1} := \min \{ s \in [ s_{i} , L ] : | u(s) - u( s_{i} ) | = r \} $ and $ x_{i+1} := u( s_{i+1} ) $.
Notice that by the denotation in section \ref{subsect:relation-l-r}, 
$ s_{i+1} = s_{i} + l( s_{i} , r ). $
\end{itemize}
(If we write $ s_{i}(r) $, we stress the dependence of this decomposition on $r$.)
In this way we get $ x_{0} , x_{1} , \ldots , x_{N(r)} \in u( [0,L] ) $ such that $ | u( s_{i-1} ) - u( s_{i} ) | = r $ for all $ i = 1 , \ldots , N(r) $ and $ | u(s) - x_{ N(r) } | < r $ for all $ s \in [ s_{ N(r) } , L ] $. 

We notice that the function $ r \mapsto N(r) $ may in general behave strangely and need not be monotone. For small $r$, however, it has the following properties.
%
%
\begin{lemma}
\label{lemma:N-prop}
Let $u\colon[0,L]\to \R^3$ be a smooth curve and assume $r$ is sufficiently small. We then have the following.
\begin{itemize}
\item The function $ r \mapsto N(r) $, as defined above, is monotonically decreasing for increasing $r$ and left-continuous. 
\item If $\hat{r}$ is a point in the jump set of $ r \mapsto N(r)$, we have $N(\hat{r})-\lim_{\rho\downarrow \hat{r}}N(\rho) =1$.
With $s_i$ defined as above, $ L = s_{N(\hat{r})}(\hat{r}) $.
\end{itemize}
\end{lemma}
\begin{proof}
Fix any $ r>0 $ such that for $ 2r $ the estimates from section~\ref{subsect:relation-l-r} are valid. Then the function in \eqref{eq:preserving-order} is strictly increasing, and therefore
\begin{equation}
\label{eq:N-monoton}
\rho \le r \Longrightarrow N( \rho ) \ge N(r).
\end{equation} 
Since $ N(r) $  takes values only in $ \N_{0} $, it is piecewise constant. 

Denote by $ l_{i} := s_{i} - s_{i-1} $ the length of the curve between $ x_{i-1} $ and $ x_{i} $ for $ i = 1 , \ldots , N(r) $ and by $ o $ its length between $ x_{ N(r) } $ and $ u(L) $.
Then we have 
\[ N(r) r \le L = \sum_{i=1}^{N(r)} l_{i} + o \le ( N(r) + 1 ) ( r + C r^{3} ). \]
Hence,
\begin{equation}
\label{eq:size-jump-N}
N(r) \in \left[ \frac{L}{ r + C r^{3} } - 1 , \frac{L}{r} \right].
\end{equation}
The length of this interval is $ 1 + \frac{ C L r }{ 1 + C r^{2} } $ so for all $r$ sufficiently small, we get at most two possible $N$. Therefore, the jumps of the function $ r \mapsto N(r) $ are of size 1.

Denote $ \nu := N(r) $. From \eqref{eq:N-monoton} it follows that for all $ \rho \le r $ the function $ \rho \mapsto s_{\nu}( \rho ) = s_{N(r)} (\rho) $ is well-defined and, being a composition of continuous functions, continuous.
Hence, 
\[ \max_{ s \in [s_{\nu}( r ),L] } | u(s) - u( s_{\nu}( r ) ) | 
= | u(L) - u( s_{\nu}( r ) ) |
= r - \varepsilon \]
for some $ \varepsilon > 0 $.
There exists a $ \delta \in ( 0 , \frac{\varepsilon}{2} ) $ such that $ | u( s_{\nu}( r ) ) - u( s_{\nu}( \rho ) ) | < \frac{\varepsilon}{2} $ for all $ \rho \in ( r - \delta , r ]. $ Therefore, for all such $ \rho $
\[ \max_{ s \in [ s_{\nu}( \rho ) ,L] } | u(s) - u( s_{\nu}( \rho ) ) | 
\le | u( s_{\nu}( r ) ) - u( s_{\nu}( \rho ) ) | + \max_{ s \in [ s_{\nu}(r) ,L] } | u(s) - u( s_{\nu}(r) ) |
< r - \tfrac{\varepsilon}{2}
< \rho, \]
and consequently, $ N(\rho) = \nu = N(r) $. The function $ r \mapsto N(r) $ is thus left-continuous.

Moreover, by \eqref{eq:size-jump-N} there exists a $ r_{1} > r $ such that $ N( \rho ) \ge \nu - 1 $ for all $ \rho \le r_{1} $.
If $ s_{\nu}(r) < L $, then 
\[ \max_{ s \in [s_{\nu-1}( r ),L] } | u(s) - u( s_{\nu-1}( r ) ) | 
= | u(L) - u( s_{\nu-1}( r ) ) |
= r + \varepsilon \]
for some $ \varepsilon \in (0,r) $. 
There exists a $ \delta \in ( 0 , \frac{\varepsilon}{2} ) $ such that $ | u( s_{\nu-1}( \rho ) ) - u( s_{\nu-1}( r ) ) | < \frac{\varepsilon}{2} $ for all $ \rho \in [ r , r + \delta ),$
and so,
\[ \max_{ s \in [ s_{\nu-1}( \rho ) ,L] } | u(s) - u( s_{\nu-1}( \rho ) ) | 
\ge | u(L) - u( s_{\nu-1}(r) ) | - | u( s_{\nu-1}( r ) ) - u( s_{\nu-1}( \rho ) ) |
> r + \tfrac{ \varepsilon }{2}
> \rho. \]
Hence, $ N( \rho ) = \nu = N(r) $. 
Therefore, in such points the function $ r \mapsto N(r) $ is also right-continuous.
\end{proof}
From Lemma~\ref{lemma:N-prop} it follows that for each $N$ sufficiently large we may  uniquely define $ r_{N} $ as 
\begin{equation}
\label{eq:def-rN}
r_{N} := \max \{ r : N(r) = N \}.
\end{equation} 
Moreover, for $ r_{N} $ the corresponding polygonal chain ends precisely in $ u(L) $.

For $ r_{N} $ we can make a more accurate estimate. With $ l_{i} = s_{i} - s_{i-1} $ being again the length of the curve between $ x_{i-1} $ and $ x_{i} $
\begin{equation*}
N r_{N} \le L = \sum_{i=1}^{N} l_{i} \le N ( r_{N} + C r_{N}^{3} ) 
\end{equation*}
or
\begin{equation}
\label{eq:l-i-bound}
r_{N} \le \frac{L}{N} \le r_{N} + C r_{N}^{3}
\quad \mbox{and} \quad
\left| l_{i} - \frac{L}{N} \right| \le C r_{N}^{3}.
\end{equation}
Then also
\begin{equation}
\label{eq:r-N-bound}
\frac{L}{N} - C \left( \frac{L}{N} \right)^{3} \le r_{N} \le \frac{L}{N}. 
\end{equation}
Moreover, by inserting $ \lambda_{N} = \frac{ N r_{N} }{L} $, we arrive at
\begin{equation*}
\frac{ \lambda_{N} L }{N} \le \frac{L}{N} \le \frac{ \lambda_{N} L }{N} + C \left( \frac{ \lambda_{N} L }{N} \right)^{3}.
\end{equation*}
Hence,
\begin{equation*}
\lambda_{N} \le 1 \le \lambda_{N} + C \lambda_{N}^{3} \left( \frac{L}{N} \right)^{2}. 
\end{equation*}
It follows
\begin{equation}
\label{eq:lambda-bound}
\lambda_{N} \ge 1 - C \left( \frac{L}{N} \right)^{2}.
\end{equation}

\section{Recovery sequence}
\label{sect:recovery}
In this section we construct a recovery sequence.
First we state a density result by which we may restrict ourselves to smooth arc-length parametrized framed rods. Then we define a suitable sequence, show that it is indeed an approximation in $L^2$ and then prove consequently convergences of the bending, torsional and penalty terms.

\subsection{Density}
As usual, it is easier to construct a recovery sequence when starting with a smooth object. In our case, we may rely on a result by Hornung \cite[Theorem 4.1]{hornung2020deformation}. For the reader's convenience, we restate this result here in the special case of a framed curve, see \cite[Section 1]{hornung2020deformation}.

\begin{theorem} \label{theo:Hornung}
Let
$ ( u , d_{2} , d_{3} ) \in W^{2,2}((0,L);\R^3) \times W^{1,2}((0,L);\R^3) \times W^{1,2}((0,L);\R^3) $
be a framed curve
and let $ Q \colon \R^{3 \times 3} \to [0,\infty) $ be continuous. Define 
\[ \F( u , d_{2} , d_{3} ) := \int_0^L Q( \mathbf{R}^{\top} \mathbf{R}' ) \ dt 
\quad \mbox{where} \quad
\mathbf{R} = ( u' , d_{2} , d_{3} ). \]
If  $ \F( u , d_{2} , d_{3} ) < \i $, then there exist framed curves $ ( u^{(n)}, d_{2}^{(n)} , d_{3}^{(n)} ) $ which are $ C^{\i} $ on $ [0,L] $ and which satisfy the following:
\begin{itemize}
    \item 
    convergence: $ ( u^{(n)}, d_{2}^{(n)} , d_{3}^{(n)} ) \to ( u, d_{2} , d_{3} ) $ strongly in $ W^{2,2}((0,L);\R^3) \times W^{1,2}((0,L);\R^3) \times W^{1,2}((0,L);\R^3) $ as $ n \to \i $,
    \item 
    convergence of energies: $ \F( u^{(n)}, d_{2}^{(n)} , d_{3}^{(n)} ) \to \F( u, d_{2} , d_{3} ) $ as $ n \to \i $,
    \item 
    boundary conditions: $ ( u^{(n)}, d_{2}^{(n)} , d_{3}^{(n)} ) \in ( u, d_{2} , d_{3} ) + W_{0}^{2,2}((0,L);\R^3) \times W_{0}^{1,2}((0,L);\R^3) \times W_{0}^{1,2}((0,L);\R^3) $ for all $ n \in \N $,
    \item 
    structure:  if $ d_{i} \cdot d_{j}'= 0 $ on $ (0,L) $ , then $ d_{i}^{(n)} \cdot (d_{j}^{(n)})'= 0 $ on $ (0,L) $ for all $ n \in \N $, denoting $u'$ by $d_1$ and ${u^{(n)}}'$ by $d_{1}^{(n)}$.
\end{itemize}
\end{theorem}

\begin{remark} \label{rem:straight_ends}
Let us stress that setting physically reasonable boundary conditions requires that we can approximate a framed arc-length parameterized curve with clamped ends by smooth ones with the same frames at their ends. By Theorem~\ref{theo:Hornung} such an approximation exists. In fact, it is even possible to choose approximating curves that are straight with constant frames on some neighborhoods of each end, as we explain below.

Let us take any of the approximating smooth framed curves $ ( u^{(n)} , d^{(n)}_2 , d^{(n)}_3 ) $ from Theorem~\ref{theo:Hornung}.
For the sake of simplicity, we write 
\[ (U,D_2,D_3) = ( u^{(n)} , d^{(n)}_2 , d^{(n)}_3 ), \quad
\mathbf{R} = ( U', D_2 , D_3 )
\quad \mbox{and} \quad
\mathbf{W} = \mathbf{R}^{\top} \mathbf{R}'.\]
We employ the results regarding degeneracy of matrix-valued maps, see \cite[Definition 3.1]{hornung2020deformation} for the definition. 
If $ \mathbf{W} $ is degenerate with respect to $ \mathfrak{so}(3) $ (the space of skew-symmetric matrices) on every subinterval $ J \subset [0,L] $, then according to \cite[Lemma 4.2]{hornung2020deformation} it is degenerate on the whole $ (0,L) $.
By \cite[Proposition 3.3 (ii)]{hornung2020deformation}, it follows $ W_{12} = W_{13} = 0 $ on $ (0,L) $, which implies $ U'' = 0 $, i.e., the curve is simply a straight line. In this case, an appropriately chosen torsion angles will yield the desired boundary behavior.

Now consider the non-degenerate case $ U'' \ne 0 $.
For every $k \in \N$ we take a function 
\[ \varphi_{k} \in C_{c}^{\i}([\tfrac{1}{k},L-\tfrac{1}{k}]) 
\quad \mbox{such that} \quad 
0 \le \varphi_{k} \le 1 
\quad \mbox{and} \quad
\varphi_{k} = 1 \mbox{ on } [\tfrac{2}{k},L-\tfrac{2}{k}], \]
and define
$ \mathbf{W}_{k} = \varphi_{k} \mathbf{R}^{\top} \mathbf{R}'. $
Then $ \mathbf{W}_{k} \to \mathbf{W} $ strongly in $ L^{2} $ as $ k \to \i $.
Since $ \mathbf{W}_{k} $ are skew-symmetric, there exist unique solutions $ \mathbf{R}_{k} \in C^{\i}([0,L]; \mathrm{SO}(3) ) $ to 
\[ \mathbf{R}_{k}' =  \mathbf{R}_{k} \mathbf{W}_{k}, \quad \mathbf{R}_{k}(0) = \mathbf{R}(0). \]
Let us introduce also the related smooth curves $ U_{k} $ such that by $ U_{k}' $ is the first column of $ \mathbf{R}_{k} $ and $ U_{k}(0) = U(0) $. By \cite[Lemma 3.4]{hornung2020deformation}, 
\[  \mathbf{R}_{k} \to \mathbf{R} \mbox{ in } W^{1,2}, \quad \mbox{and consequently,} \quad U_{k} \to U \mbox{ in } W^{2,2}.\]
Since $ \mathbf{W}_{k} = 0 $ on $ [0,\frac{1}{k}] \cup [L-\frac{1}{k},L] $, the frame $ \mathbf{R}_{k} $ is constant on these intervals.
However, possibly $ U_{k}(L) \ne U(L) $ and $ \mathbf{R}_{k}(L) \ne \mathbf{R}(L) $.

By the consideration above, since $ U'' \ne 0 $, there exists a non-empty open interval $ J \subset (0,L) $ such that $ \mathbf{R} $ is not degenerate with respect to $ \mathfrak{so}(3) $ on $J$. 
Since $ C^{\i}_{c}( J ; \mathfrak{so}(3) ) $ is a dense subspace of 
$ \{ \hat{ \mathbf{W} } \in L^{2}( (0,L) ; \mathfrak{so}(3) ) : \mathrm{supp} \hat{ \mathbf{W} } \subset J \} $, 
we may apply \cite[Theorem 3.2]{hornung2020deformation}.
There exist
\begin{itemize}
    \item a finite dimensional subspace $ E $ of $ C^{\i}_{c}( J ; \mathfrak{so}(3) ) $ and
    \item for each $ k \in \N $ a function $ \hat{ \mathbf{W} }_{k}  \in E $
\end{itemize}
such that $ \hat{ \mathbf{W} }_{k} \to 0 $ as $ k \to \i $ and that the framed smooth curves $ \hat{U}_{k} $, determined by 
\[ \hat{ \mathbf{R} }_{k}' =  \hat{ \mathbf{R} }_{k} ( \mathbf{W}_{k} + \hat{ \mathbf{W} }_{k} ),
\quad \hat{ \mathbf{R} }_{k}(0) = \mathbf{R}(0),
\quad \hat{U}_{k}(0) = U(0) \]
analogously as above, fulfill also
\[ \hat{ \mathbf{R} }_{k}(L) = \mathbf{R}(L) 
\quad \mbox{and} \quad 
\hat{U}_{k}(L) = U(L). \]
Since $ \mathbf{W}_{k} + \hat{ \mathbf{W} }_{k} \to \mathbf{W} $ in $ L^2 $ as $ k \to \i $,
we also have $ \hat{U}_{k} \to U $ in $ W^{2,2} $ and $ \hat{ \mathbf{R} }_{k} \to \mathbf{R} $ in $ W^{1,2} $. 
For $k$ large enough we may choose $ J \subset ( \frac{1}{k},L-\frac{1}{k} ) $. Hence, $ \hat{U}_{k} $ are for these $k$ straight on some neighborhoods of both ends with constant frames.

We need to comment also on the convergence of the energies. 
%
%
%
%
Since $ \mathbf{W} $ is smooth, its values on $ [0,L] $ lie in some closed ball $K$, and by the definition, the same holds for $ \mathbf{W}_{k} $.
Since $ Q $ is continuous, it is uniformly continuous on $K$, and therefore
\[ \left| \int_{0}^{L} Q( \mathbf{R}_{k}^{\top} \mathbf{R}_{k}' ) \ dt - \int_{0}^{L} Q( \mathbf{R}^{\top} \mathbf{R}' ) \ dt \right|
\le \int_{0}^{L} | \tilde{Q}( \mathbf{W}_{k} ) - \tilde{Q}( \mathbf{W} ) | \ dt
\le \int_{0}^{L} C | \varphi_{k} \mathbf{W} - \mathbf{W} | \ dt
\to 0. \]
\end{remark}

\subsection{Construction of the sequence}
By Theorem~\ref{theo:Hornung} we may restrict ourselves to the smooth case. Again, employing the Bishop frame, let us take any pair $ ( u , \theta ) \in C^{\i}( [0,L] ; \R^{3} ) \times C^{\i}( [0,L] ) $ with $ |u'| \equiv 1 $.
We first approximate the curve $u$ by a polygonal curve with $N$ line segments of the same length.
Let us fix a sufficiently large $ N $  so that there exists $ r_{N} $ as in \eqref{eq:def-rN}. 
For this $ r_{N} $, we employ the construction from subsection~\ref{subsect:app-equi} to obtain
\[ 0 = s_{0}^{N} < s_{1}^{N} < \ldots < s_{N}^{N} = L 
\quad \mbox{such that} \quad
| u( s_{i}^{N} ) - u( s_{i-1}^{N} ) | = r_{N}
\mbox{ for all $ i = 1 , \ldots , N $.} \]
This construction is illustrated in Figure~\ref{fig:equi-polygon}.
\begin{figure}[ht!]
\includegraphics[width=.9\textwidth]{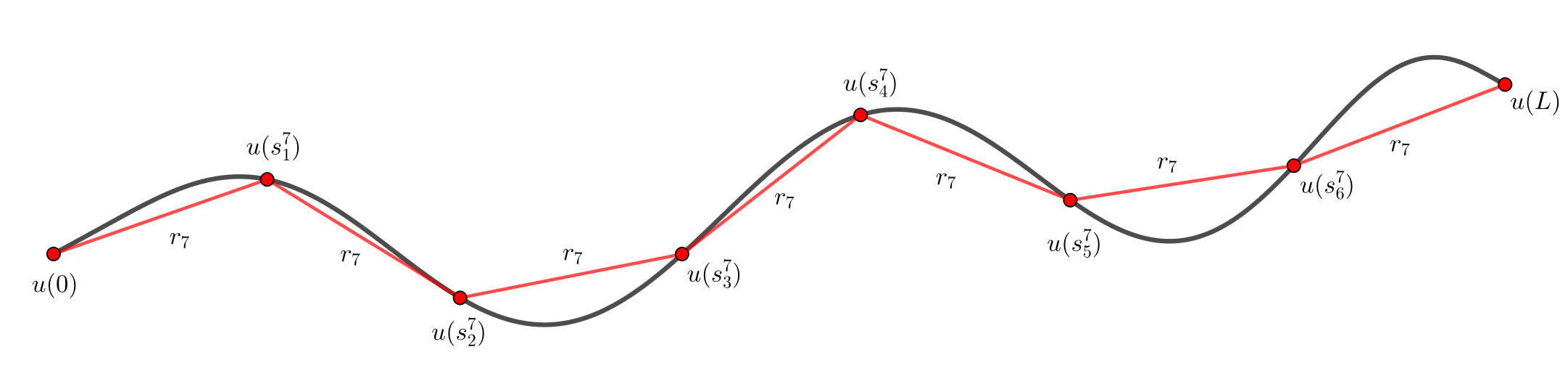}	
\caption{In subsection~\ref{subsect:app-equi} we constructed for each $N$ sufficiently large an approximation of the curve $u$ by a polygonal line with the same endpoints having $N$ segments of equal length $r_{N}$. In the picture, a situation with $ N = 7 $ segments is shown.}
\label{fig:equi-polygon}
\end{figure}

We define the related framed discrete rod $ ( X_{N} , \Phi_{N} ) \in ( \R^{3} )^{N+1} \times \R^{N} $ as 
\[ X_{N} := ( x_{0}^{N} , x_{1}^{N} , \ldots , x_{N}^{N} ) 
\quad \mbox{where} \quad
x_{i}^{N} := u( s_{i}^{N} )
\quad
\mbox{for all $ i = 0 , \ldots , N $} \]
and
\[ \Phi_{N} := ( \varphi_{1}^{N} , \ldots , \varphi_{N}^{N} ) 
\quad \mbox{where} \quad
\varphi_{i}^{N} := \theta( \tfrac{ s_{i-1}^{N} + s_{i}^{N} }{2}  )
\quad
\mbox{for all $ i = 1 , \ldots , N $}. \]
Moreover, define correspondingly, as in subsection~\ref{subsect:assign-spline}, also $ \tau_{i}^{N} $ as in \eqref{eq:def-tau}, $ \lambda_{N} := \lambda( X_{N} ) $, $ t^{N}_{i} := \lambda_{N}^{-1} \tau^{N}_{i} $, $ \eta_{N} := \eta^{ X_{N} } $, $ y_{N} := y^{ X_{N} } $, $ \zeta_{N} := \zeta^{ X_{N} , \Phi_{N} } $ and $ z_{N} := z^{ X_{N} , \Phi_{N} } $.
Notice that
\[ \ell( X_{N} ) = \lambda_{N} L = N r_{N}, \quad
\tau_{i}^{N} = (i-\tfrac{1}{2}) r_{N}, \quad
t_{i}^{N} = (i-\tfrac{1}{2}) \tfrac{L}{N}. \]
We claim that $ ( y_{N} , z_{N} ) $ is a recovery sequence for $ ( u , \theta ) $.

For the definition of $ \eta_{N} $, see \eqref{eq:def-eta}, we have to determine the cubic functions $ S^{N}_{i} $, $ i = 1, \ldots, N-1 $. 
Since for every $ i \in \{ 1 , \ldots , N-1 \} $ 
\[ \tau^{N}_{i+1} - \tau^{N}_{i} = | x^{N}_{i+1} - x^{N}_{i} | = r_{N}, \]
it has to fulfil
\[ S^{N}_{i}(0) = \overline{ x^{N}_{i} }, \quad
S^{N}_{i}( r_{N} ) = \overline{ x^{N}_{i+1} }, \quad
( S^{N}_{i} )'(0) = \frac{ x^{N}_{i} - x^{N}_{i-1} }{ r_{N} }, \quad
( S^{N}_{i} )'( r_{N} ) = \frac{ x^{N}_{i+1} - x^{N}_{i} }{ r_{N} }. \]
As the conditions \eqref{eq:spline-1}, \eqref{eq:spline-2} and \eqref{eq:spline-3} are fulfilled, we arrive at
\begin{equation}
\label{eq:quadratic-spline}
    S^{N}_{i}(\tau) = \frac{ x^{N}_{i+1} - 2 x^{N}_{i} + x^{N}_{i-1} }{ 2 r_{N}^{2} } \ \tau^{2} + \frac{ x^{N}_{i} - x^{N}_{i-1} }{ r_{N} } \ \tau + \overline{ x^{N}_{i} }.
\end{equation}
Let us denote
\begin{equation*}
\Delta_{ r_{N} } x^{N}_{i} := \frac{ x^{N}_{i+1} - 2 x^{N}_{i} + x^{N}_{i-1} }{ r_{N}^{2} }.
\end{equation*} 
Then
\begin{equation}
\label{eq:eta-prime-prime}
\eta_{N}''( \tau ) 
= \left\{ \begin{array}{cl}
0 , & \tau \le \tau_{1}, \\
\Delta_{ r_{N} } x^{N}_{i}, & \tau_{i} \le \tau \le \tau_{i+1} \mbox{ for some } i \in \{ 1 , \ldots, N-1 \}, \\
0, & \tau_{N} \le \tau.
\end{array} \right.
\end{equation} 

\subsection{Convergence $ ( y_{N} , z_{N} ) \to ( u , \theta ) $ in $ L^{2} $}
Clearly $ y_{N}(0) = u(0) $ and $ y_{N}(L) = u(L) $. Let us additionally denote $ t^{N}_{0} := 0 $ and $ t^{N}_{N+1} := L $. 

Let us fix arbitrary $ s \in (0,L) $ and estimate $ | y_{N}(s) - u(s) | $ and $ | z_{N}(s) - \theta(s) | $.
There exists a unique $ i \in \{ 0 , \ldots , N \} $ such that $ s \in [ t^{N}_{i} , t^{N}_{i+1} ) $. 

For the first expression, let us split
\[ | y_{N}(s) - u(s) | 
\le | y_{N}(s) - x^{N}_{i} | + | u(s) - x^{N}_{i}   |. \]
Since $ x^{N}_{i} = u( s^{N}_{i} ) $ and
\[ 0 = t^{N}_{0} = s^{N}_{0} < t^{N}_{1} < s^{N}_{1} < \ldots < t^{N}_{N} < s^{N}_{N} = t^{N}_{N+1} = L,\]
we have
\[ | u(s) - x^{N}_{i} |= | u(s) - u( s^{N}_{i} ) | \le \| u' \|_{\i} | s - s^{N}_{i} | \le \| u' \|_{\i} \frac{L}{N}. \]
If $ 1 \le i \le N-1 $, we have $ y_{N}(s) = S^{N}_{i}( \lambda_{N} s - \tau^{N}_{i}  ) $.
By \eqref{eq:quadratic-spline} for $ \tau \in [0,r_{N}] $,
\begin{eqnarray*}
| S^{N}_{i}( \tau  ) - x^{N}_{i} | 
& = & \left| \frac{ x^{N}_{i+1} - 2 x^{N}_{i} + x^{N}_{i-1} }{ 2 r_{N}^{2} } \ \tau^{2} + \frac{ x^{N}_{i} - x^{N}_{i-1} }{ r_{N} } \ \tau + \frac{ - x^{N}_{i} + x^{N}_{i-1} }{2} \right| \\
& \le & \frac{ | x^{N}_{i+1} - x^{N}_{i} | + | - x^{N}_{i} + x^{N}_{i-1} |}{ 2 r_{N}^{2} } \ \tau^{2} + \frac{ | x^{N}_{i} - x^{N}_{i-1} | }{ r_{N} } \ \tau + \frac{ | - x^{N}_{i} + x^{N}_{i-1} | }{2} \\
& = & \frac{ 2 r_{N} }{ 2 r_{N}^{2} } \ \tau^{2} + \frac{ r_{N} }{ r_{N} } \ \tau + \frac{ r_{N} }{2} \\
& \le & \frac{5}{2} r_{N}.
\end{eqnarray*}
If $i=0$, then by \eqref{eq:def-eta}
\[ | y_{N}(s) - x^{N}_{0} | 
= | \eta_{N}( \lambda_{N} s ) - x^{N}_{0} | 
= \left| x^{N}_{0} + \lambda_{N} s \frac{ x^{N}_{1} - x^{N}_{0} }{ | x^{N}_{1} - x^{N}_{0} | } - x^{N}_{0} \right|
\le \frac{ r_{N} }{2}, \]
and analogously for $i=N$. Hence
\[ | y_{N}(s) - u(s) | 
\le \| u' \|_{\i} \frac{L}{N} + \frac{5}{2} r_{N}. \]

Now the second expression. If $ i \not\in \{ 0 , N \} $, we have
\begin{eqnarray*}
| z_{N}(s) - \theta(s) | 
& = & \left| \varphi^{N}_{i} + \frac{ \varphi^{N}_{i+1} - \varphi^{N}_{i} }{ t^{N}_{i+1} - t^{N}_{i} }( s - t^{N}_{i} ) - \theta(s) \right| \\
& \le & | \varphi^{N}_{i+1} - \varphi^{N}_{i} | + | \varphi^{N}_{i}  - \theta(s) | \\
& \le & | \theta( \tfrac{ s_{i}^{N} + s_{i+1}^{N} }{2}  ) - \theta( \tfrac{ s_{i-1}^{N} + s_{i}^{N} }{2}  ) | + | \theta( \tfrac{ s_{i-1}^{N} + s_{i}^{N} }{2}  )  - \theta(s) | \\
& \le & \| \theta' \|_{L^{\infty}} \frac{ s_{i+1}^{N} - s_{i-1}^{N} }{2} 
+ \| \theta' \|_{L^{\infty}} \left| \frac{ s_{i-1}^{N} + s_{i}^{N} }{2}  - s \right| \\
& \le & \| \theta' \|_{L^{\infty}} \left( \frac{ 3 L }{2 N } 
+  \frac{L}{N} \right)  \\
&  =  & \| \theta' \|_{L^{\infty}}  \frac{ 5 L }{2 N }.
\end{eqnarray*} 
It can be easily seen that the same bound is valid also for the remaining two values of $i$.

Thus we have shown that $ y_{N} \to u $ and $ z_{N} \to \theta $ uniformly $u$ as $ N \to \i $.

\subsection{Convergence of bending energy}
\label{subsect:bend}
Let us now show that $ \| y_{N}'' \|_{L^{2}} \to \| u'' \|_{L^{2}}. $
%
From \eqref{eq:eta-prime-prime}, it follows
\begin{equation*}
\| y_{N}'' \|_{ L^{2} }  
= \lambda_{N}^{3} \int_{0}^{ \lambda_{N} L} | ( \eta^{X} )''( \tau ) |^{2} \ d \tau 
= \lambda_{N}^{3} \sum_{i=1}^{N-1} r_{N} | \Delta_{ r_{N} } x^{N}_{i} |^{2}.
\end{equation*}
By employing $ N r_{N} = \lambda_{N} L $, we get
\begin{equation*}
\| y_{N}'' \|_{ L^{2} }  
= \lambda_{N}^{3} r_{N} \sum_{i=1}^{N-1} \left| \frac{ u( s^{N}_{i-1} ) + u( s^{N}_{i+1} ) - 2 u( s^{N}_{i} ) }{ r_{N}^{2} } \right|^{2}
= \frac{L}{N} \sum_{i=1}^{N-1} \left| \frac{ u( s^{N}_{i-1} ) + u( s^{N}_{i+1} ) - 2 u( s^{N}_{i} ) }{ ( \frac{L}{N} )^{2} } \right|^{2}.
\end{equation*}
This expression resembles closely the approximation of $ \| u''\|_{ L^{2} }^{2} $ via equidistant partition of the interval
\begin{equation*} 
{\B}_{N}( u ) 
= \frac{L}{N} \sum_{i=1}^{N-1} \left| \frac{ u( (i-1) \frac{L}{N} ) + u( (i+1) \frac{L}{N} ) - 2 u( i \frac{L}{N} ) }{ ( \frac{L}{N} )^{2} } \right|^{2},
\end{equation*}
see, e.g., \cite{braides2000discrete}. Knowing that $ {\B}_{N}( u ) \to \| u''\|_{ L^{2} }^{2} $, our goal is to prove
\[ | \| y_{N}'' \|_{ L^{2} } - {\B}_{N}( u ) | \to 0. \]
For each $i$ we may write
\begin{equation*} 
u( s^{N}_{i} ) - u( i \tfrac{L}{N} ) 
= u'( i \tfrac{L}{N} ) \big( t^{N}_{i} - i \tfrac{L}{N} \big) + u''( \zeta^{N}_{i} ) \big(  t^{N}_{i} - i \tfrac{L}{N} \big)^{2} 
\end{equation*}
for some $ \zeta^{N}_{i} \in [ i \frac{L}{N} , s^{N}_{i} ] \cup [ s^{N}_{i} , i \frac{L}{N} ] $.
Then
\begin{eqnarray*} 
D_{i}^{N}
&:= & \Big( u( s^{N}_{i-1} ) + u( s^{N}_{i+1} ) - 2 u( s^{N}_{i} ) \Big)
- \Big( u( (i-1) \tfrac{L}{N} ) + u( (i+1) \tfrac{L}{N} ) - 2 u( i \tfrac{L}{N} ) \Big) \\
& = & u'( (i-1) \tfrac{L}{N} ) \big( s^{N}_{i-1} - (i-1) \tfrac{L}{N} \big) 
 + u'( (i+1) \tfrac{L}{N} ) \big( s^{N}_{i+1} - (i+1) \tfrac{L}{N} \big) 
 - 2 u'( i \tfrac{L}{N} ) \big( s^{N}_{i} - i \tfrac{L}{N} \big)  \\
&   & + u''( \zeta^{N}_{i-1} ) \big(  s^{N}_{i-1} - (i-1) \tfrac{L}{N} \big)^{2} + u''( \zeta^{N}_{i+1} ) \big(  s^{N}_{i+1} - (i+1) \tfrac{L}{N} \big)^{2} - 2 u''( \zeta^{N}_{i} ) \big(  s^{N}_{i} - i \tfrac{L}{N} \big)^{2}.
\end{eqnarray*}
By employing the bounds for $ l_{i}^{N} = s^{N}_{i} - s^{N}_{i-1} $ from \eqref{eq:l-i-bound}, we obtain
\begin{equation}
\label{eq:t-i-N}
| s^{N}_{i} - i \tfrac{L}{N} | 
= | l^{N}_{1} + \ldots + l^{N}_{i} - i \tfrac{L}{N} |
\le \sum_{j=1}^{i} | l^{N}_{j}  - \tfrac{L}{N} |
\le i C ( \tfrac{L}{N} )^{3} 
\le N C ( \tfrac{L}{N} )^{3}.
\end{equation}
Thus we can estimate the terms with second derivatives by
\begin{multline*} 
\Big| u''( \zeta^{N}_{i-1} ) \big(  s^{N}_{i-1} - (i-1) \tfrac{L}{N} \big)^{2} + u''( \zeta^{N}_{i+1} ) \big(  s^{N}_{i+1} - (i+1) \tfrac{L}{N} \big)^{2} - 2 u''( \zeta^{N}_{i} ) \big(  s^{N}_{i} - i \tfrac{L}{N} \big)^{2} \Big| \le \\
\le 3 \| u'' \|_{\i} \Big( N C ( \tfrac{L}{N} )^{3} \Big)^{2}.
\end{multline*}
For the terms with first derivatives, it holds that
\begin{align*} 
& \Big| u'( (i-1) \tfrac{L}{N} ) \big( s^{N}_{i-1} - (i-1) \tfrac{L}{N} \big) 
 + u'( (i+1) \tfrac{L}{N} ) \big( s^{N}_{i+1} - (i+1) \tfrac{L}{N} \big) 
 - 2 u'( i \tfrac{L}{N} ) \big( s^{N}_{i} - i \tfrac{L}{N} \big) \Big| \\
& \le  \ | s^{N}_{i} - i \tfrac{L}{N} |
\Big| u'( (i-1) \tfrac{L}{N} ) + u'( (i+1) \tfrac{L}{N} ) - 2 u'( i \tfrac{L}{N} ) \Big| + \\
&  \qquad + \Big| u'( (i-1) \tfrac{L}{N} ) \big( s^{N}_{i-1} - s^{N}_{i} + \tfrac{L}{N} \big) 
 + u'( (i+1) \tfrac{L}{N} ) \big( s^{N}_{i+1} - s^{N}_{i} - \tfrac{L}{N} \big) \Big| \\
& \le \  N C ( \tfrac{L}{N} )^{3} \cdot 2 \| u'' \|_{\i} \tfrac{L}{N} + 2 C ( \tfrac{L}{N} )^{3} \\
& =  \ ( \tfrac{L}{N} )^{3} \cdot 2 C ( L \| u'' \|_{\i} + 1 ).
\end{align*}
We have used that we may write
\[ \Big| u'( (i-1) \tfrac{L}{N} ) + u'( (i+1) \tfrac{L}{N} ) - 2 u'( i \tfrac{L}{N} ) \Big|
= \Big| u''( \hat{\zeta}_{i} ) ( - \tfrac{L}{N} ) + u''( \hat{\zeta}_{i+1} ) ) \tfrac{L}{N} \Big|
\le 2 \| u'' \|_{\i} \tfrac{L}{N} \]
for some $ \hat{\zeta}_{i} \in [ (i-1) \tfrac{L}{N} , i \tfrac{L}{N} ] $ 
and $ \hat{\zeta}_{i+1} \in [ i \tfrac{L}{N} , (i+1) \tfrac{L}{N} ] $, 
and that
\[ | s^{N}_{i+1} - s^{N}_{i} - \tfrac{L}{N} | = | l^{N}_{i+1} - \tfrac{L}{N} | \le C ( \tfrac{L}{N} )^{3}. \]
Hence,
\begin{equation*}
\left| \frac{ u( s^{N}_{i-1} ) + u( s^{N}_{i+1} ) - 2 u( s^{N}_{i} ) }{ ( \frac{L}{N} )^{2} } \right|^{2}
= \left| \frac{ u( (i-1) \frac{L}{N} ) + u( (i+1) \frac{L}{N} ) - 2 u( i \frac{L}{N} ) + D_{i}^{N} }{ ( \frac{L}{N} )^{2} } \right|^{2}
\end{equation*}
with
\begin{equation*}
| D^{N}_{i} | 
\le 2 C ( L \| u'' \|_{\i} + 1 ) \left( \frac{L}{N} \right)^{3}
+ 3 C L^2 \| u'' \|_{\i} \left( \frac{L}{N} \right)^{4}. 
\end{equation*}
Hence,
\begin{align*}
\lim_{ N \to \i } \| y_{N}'' \|_{ L^{2} }  
& = \lim_{ N \to \i } \frac{L}{N} \sum_{i=1}^{N-1} \left| \frac{ u( (i-1) \frac{L}{N} ) + u( (i+1) \frac{L}{N} ) - 2 u( i \frac{L}{N} ) + D^{N}_{i}  }{ ( \frac{L}{N} )^{2} } \right|^{2} \\
& = \lim_{ N \to \i } \frac{L}{N} \sum_{i=1}^{N-1} \left| \frac{ u( (i-1) \frac{L}{N} ) + u( (i+1) \frac{L}{N} ) - 2 u( i \frac{L}{N} )  }{ ( \frac{L}{N} )^{2} } \right|^{2} \\
& = \lim_{ N \to \i } \mathcal{B}_{N}(u) \\
& = \| u'' \|_{ L^{2} }^{2}.
\end{align*}

\subsection{Convergence of torsional energy}
Now, we will in a similar way show $ \| z_{N}' \|_{ L^{2} } \to \| \theta' \|_{ L^{2} } $.
First we notice that 
\begin{equation*}
\zeta_{N}'( \tau ) 
= \left\{ \begin{array}{cl}
0 , & \tau \le \tau_{1}, \\
\frac{ \varphi_{i+1}^{N} - \varphi_{i}^{N} }{ r_{N} }, & \tau_{i} \le \tau \le \tau_{i+1} \mbox{ for some } i \in \{ 1 , \ldots, N-1 \}, \\
0, & \tau_{N} \le \tau,
\end{array} \right.
\end{equation*} 
and therefore,
\begin{equation*}
\| z_{N}' \|_{ L^{2} }  
= \lambda(X) \int_{0}^{ \lambda(X) L} ( \zeta^{X, \Phi} )'( \tau )^{2} \ d \tau
= \lambda_{N} \sum_{i=1}^{N-1} r_{N} \left| \frac{ \varphi_{i+1}^{N} - \varphi_{i}^{N} }{ r_{N} } \right|^{2}.
\end{equation*}
By employing $ N r_{N} = \lambda_{N} L $, we get
\begin{equation*}
\| z_{N}' \|_{ L^{2} }  
= \lambda_{N} r_{N} \sum_{i=1}^{N-1} \left| \frac{ \varphi_{i+1}^{N} - \varphi_{i}^{N} }{ r_{N} } \right|^{2}
= \frac{L}{N} \sum_{i=1}^{N-1} \left| \frac{ \theta( \frac{ s_{i+1}^{N} + s_{i}^{N} }{2}  ) - \theta( \frac{ s_{i}^{N} + s_{i-1}^{N} }{2} ) }{ \frac{L}{N} } \right|^{2}.
\end{equation*}
Analogously to section~\ref{subsect:bend}, we reduce this problem to the case with the equidistant partition.
It is known \cite[Chapter 4]{MR1968440} that
\begin{equation*} 
\mathcal{T}_{N}( \theta ) 
:= \frac{L}{N} \sum_{i=1}^{N-1} \left| \frac{ \theta( (i+\frac{1}{2}) \frac{L}{N} ) - \theta( (i-\frac{1}{2}) \frac{L}{N} )  }{ \frac{L}{N} } \right|^{2} \to \| \theta' \|_{ L^{2} }^{2}. 
\end{equation*}
For each $i$ we may write
\begin{equation*} 
\theta( \tfrac{ s_{i-1}^{N} + s_{i}^{N} }{2} ) - \theta( (i-\tfrac{1}{2}) \tfrac{L}{N} ) 
= \theta'( \omega^{N}_{i} ) \big( \tfrac{ s_{i-1}^{N} + s_{i}^{N} }{2} - (i-\tfrac{1}{2}) \tfrac{L}{N} \big) 
\end{equation*}
for some $ \omega^{N}_{i} \in [ \tfrac{ s_{i-1}^{N} + s_{i}^{N} }{2} , (i-\tfrac{1}{2}) \tfrac{L}{N} ] \cup [ (i-\tfrac{1}{2}) \tfrac{L}{N} , \tfrac{ s_{i-1}^{N} + s_{i}^{N} }{2} ] $.
Then
\begin{eqnarray*} 
d_{i}^{N}
&:= & \Big( \theta( \tfrac{ s_{i+1}^{N} + s_{i}^{N} }{2}  ) - \theta( \tfrac{ s_{i}^{N} + s_{i-1}^{N} }{2} ) \Big)
- \Big( \theta( (i+\tfrac{1}{2}) \tfrac{L}{N} ) - \theta( (i-\tfrac{1}{2}) \tfrac{L}{N} ) \Big) \\
& = & \theta'( \omega^{N}_{i+1} ) \big( \tfrac{ s_{i+1}^{N} + s_{i}^{N} }{2} - (i+\tfrac{1}{2}) \tfrac{L}{N} \big)
- \theta'( \omega^{N}_{i} ) \big( \tfrac{ s_{i-1}^{N} + s_{i}^{N} }{2} - (i-\tfrac{1}{2}) \tfrac{L}{N} \big).
\end{eqnarray*}
Since by \eqref{eq:t-i-N} $ | s^{N}_{i} - i \tfrac{L}{N} | \le N C ( \tfrac{L}{N} )^{3} $, 
we have
\begin{equation*}
\left| \tfrac{ s_{i-1}^{N} + s_{i}^{N} }{2} - (i-\tfrac{1}{2}) \tfrac{L}{N} \right| 
= \left| \tfrac{ s_{i-1}^{N} - (i-1) \tfrac{L}{N} }{2} + \tfrac{ s_{i}^{N} - i \tfrac{L}{N} }{2} \right|
\le N C ( \tfrac{L}{N} )^{3}, 
\end{equation*}
and
\[ d_{i}^{N} \le 2 \| \theta' \|_{\i} C L ( \tfrac{L}{N} )^{2}. \]
Thus,
\begin{equation*}
\| z_{N}' \|_{ L^{2} }^2
= \frac{L}{N} \sum_{i=1}^{N-1} \left| \frac{ \theta( (i+\frac{1}{2}) \frac{L}{N} ) - \theta( (i-\frac{1}{2}) \frac{L}{N} ) + d_{i}^{N} }{ \frac{L}{N} } \right|^{2} 
\to \| \theta' \|_{ L^{2} }^2. 
\end{equation*}

\begin{remark} \label{rem:clamped_frame_recovery}
If clamped boundary conditions were desired, we now observe that once the curve was approximated according to Theorem \ref{theo:Hornung} and Remark \ref{rem:straight_ends}, we obtain the correct frames at the boundaries once $N$ is large enough. Indeed, when at least one interior discretization point lies on the straight part of the curve, the corresponding framed spline curve will have a constant frame near the boundary.
\end{remark}

\subsection{Conclusion}
\label{subsect:conclusion}
Lastly, as by \eqref{eq:lambda-bound} and \eqref{eq:r-N-bound} 

\[ 1 \ge \lambda_{N} \ge 1 - C \left( \tfrac{L}{N} \right)^{2} 
\quad \mbox{and} \quad
|  x_{i}^{N} - x_{i-1}^{N} | = r_{N} \le \tfrac{L}{N}, \]
we have
\[ \F_{N}^{\rm pen}(X_{N}) 
= N^{\alpha} | \lambda_{N} - 1 | + N^{\beta} \max_{i} |  x_{i}^{N} - x_{i-1}^{N} | 
\le N^{\alpha} \cdot C \left( \tfrac{L}{N} \right)^{2} + N^{\beta} \cdot \tfrac{L}{N}.
\]
Now, it becomes apparent why we chose $ \alpha < 2 $ and $ \beta < 1 $, as then
\[ \F_{N}^{\rm pen}(X_{N}) \to 0. \]
Hence, indeed
\begin{align*}
\lim_{ N \to \i } \F_{N}( y_{N} , z_{N} )
& = \lim_{ N \to \i } \big( \F^{\rm bend}( y_{N} ) + \F^{\rm tor}( z_{N} ) + \F^{\rm pen}_{N}( X_{N} ) \big) \\
& = \F^{\rm bend}( u ) + \F^{\rm tor}( \theta ) + 0 \\
& = \F( u , \theta ).
\end{align*} 

\section*{Acknowledgements}
PD acknowledges partial support from the DFG via project 441523275 in SPP 2256.
MJ was supported in part by Grant P1-0222 of the Slovenian Research Agency.
MJ acknowledges the support and hospitality of the University of Freiburg during his stay there in February 2022.
\printbibliography

\end{document}